\documentclass[10pt]{amsart}
\usepackage{amsmath,amssymb,amscd,dsfont}
\usepackage{graphicx,color}
\newtheorem{theorem}{Theorem}[section]
\newtheorem{corollary}[theorem]{Corollary}

\newtheorem{proposition}[theorem]{Proposition}
\newtheorem{hypothesis}[theorem]{Hypothesis}

\theoremstyle{definition}
\newtheorem{definition}[theorem]{Definition}

\newtheorem{remark}[theorem]{Remark}

\numberwithin{equation}{section}

\newcommand{\R}{\mathds{R}}
\newcommand{\C}{\mathds{C}}
\newcommand{\N}{\mathds{N}}

\newcommand{\al}{\alpha}
\newcommand{\be}{\beta}
\newcommand{\gm}{\gamma}
\newcommand{\lm}{\lambda}

\newcommand{\beq}{\begin{equation}}
\newcommand{\eeq}{\end{equation}}
\newcommand{\beqnt}{\begin{equation*}}
\newcommand{\eeqnt}{\end{equation*}}
\DeclareMathOperator{\rd}{d}
\newcommand{\set}[2]{\left\{  #1 : #2 \right\}  }

\newcommand{\SSet}[2]{\left\{  #1 : #2 \right\}  }

\newcommand\re{\mathrm{Re }}

\newcommand\I{\mathrm{i}}

\newcommand{\dom}{\mathcal{D}}
\newcommand\Ran{\mathrm{ran }}
\renewcommand\H{\mathcal{H}}

\newcommand\A{\mathcal{A}}

\newcommand{\LL}{\operatorname{L}}

\DeclareMathOperator*{\nlim}{n\,-lim}

\DeclareMathOperator{\diag}{diag}
\newcommand\rref[1]{{\rm \ref{#1}}}

\begin{document}

\title[Block-Diagonalization of Operators with Gaps]
{Block-Diagonalization of Operators with Gaps, with Applications to Dirac operators}

\author{Jean-Claude Cuenin}
\address{Departement of Mathematics, Imperial College London, London SW7 2AZ, UK}
\email{j.cuenin@imperial.ac.uk}

\thanks{The author gratefully acknowledges the support of Schweizer Nationalfonds, SNF, Grant No.\ $200020\_130184$.}

\begin{abstract}
We present new results on the block-diagonalization of Dirac operators on three-dimensional Euclidean space with unbounded potentials. Classes of admissible potentials include electromagnetic potentials with strong Coulomb singularities and more general matrix-valued potentials, even non-self-adjoint ones. For the Coulomb potential, we achieve an exact diagonalization up to nuclear charge $Z=124$ and prove the convergence of the Douglas-Kroll-He\ss\ approximation up to $Z=62$, thus improving the upper bounds $Z=93$ and $Z=51$, respectively, by H.\ Siedentop and E.\ Stockmeyer considerably.
These results follow from abstract theorems on perturbations of spectral subspaces of operators with gaps, which are based on a method of H.\ Langer and C.\ Tretter and are also of independent interest.
\end{abstract}

\maketitle

\setcounter{tocdepth}{1}
\tableofcontents

\section{Introduction}

The Dirac operator on $\R^3$ governing the motion of a relativistic particle of half-integer spin in the presence of an external electromagnetic field is given by (in units where the reduced Planck constant, the velocity of light and the particle mass are equal to one)
\beq\label{electromagnetic Dirac}
H=\begin{pmatrix}1+\Phi &-\I\,\boldsymbol{\sigma}\cdot(\nabla-\I\mathbf{A})\\-\I\,\boldsymbol{\sigma}\cdot(\nabla-\I\mathbf{A})&-1+\Phi \end{pmatrix}.
\eeq
Here, $\boldsymbol{\sigma}:=(\sigma_1,\sigma_2,\sigma_3)$ is a formal vector whose components are the Pauli-matrices 
\beqnt
 \sigma_1=\begin{pmatrix}
           0&1\\1&0
          \end{pmatrix},\quad
 \sigma_2=\begin{pmatrix}
           0&-\I\\ \I&0
          \end{pmatrix},\quad
 \sigma_3=\begin{pmatrix}
           1&0\\0&-1
          \end{pmatrix},
\eeqnt
$\Phi:\R^3\to\R$ is the electric potential and $\mathbf{A}:\R^3\to\R^3$ is the magnetic vector potential, which determine the electric and magnetic field (uniquely up to a choice of gauge) by virtue of
\[
 \mathbf{E}=\nabla\Phi,\quad \mathbf{B}=\operatorname{curl}\mathbf{A}.
\]
We regard $H$ as an unbounded operator in the Hilbert space of square-integrable functions

\beq\label{natural decomposition}
\H=L^2\left(\R^3,\C^2\right)\oplus L^2\left(\R^3,\C^2\right).
\eeq

It is well known that the Dirac operator is not bounded from below, giving rise to an infinite ``sea" of unphysical negative energy states. In the context of Dirac's hole theory, the Pauli exclusion principle is invoked in order to ``fill the sea", i.e.\ declare the states of negative energy as already occupied and restrict $H$ to its positive spectral subspace. While hole theory turned out to be unsatisfactory from a theoretical point of view and was ultimately made obsolete by quantum electrodynamics, it still has its merits in atomic physics and quantum chemistry, where the energy scale is well below the threshold for particle creation and annihilation.

In the field-free case ($\Phi=0$, $\mathbf{A}=0$) the now famous Foldy-Wouthuysen transformation \cite{FW50} may be invoked to decouple the positive and negative spectral subspaces. However, the original method proposed by Foldy and Wouthuysen in the presence of external fields makes use of an ill-defined expansion in the inverse speed of light and has to be discarded, see \cite{Th}. 
For Dirac operators in a purely magnetic field ($\Phi=0$), supersymmetric methods have been employed by Thaller \cite{Tha87, Th} to construct an exact transformation. However, these methods do not apply to electric potentials, which are of paramount importance from the physical point of view. 
In \cite{LT}, Langer and Tretter developed an abstract method for diagonalization of block operator matrices based on indefinite inner product spaces, which yields an exact transformation of the Dirac operator for bounded electric potentials of norm less than one. 
Siedentop and Stockmeyer \cite{SS06} proved the existence of an exact transformation for the Dirac operator with Coulomb potential, as well as the convergence of an approximate block-diagonalization, known as the Douglas-Kroll-He\ss\ (DKH) method. The latter, proposed by Douglas and Kroll in their seminal paper \cite{DK74}, consists of a an iterative scheme which decouples the positive and negative spectral subspaces up to any given order in the coupling constant of the potential. Its usefulness for quantum chemical implementations was first realized by B.A. He\ss\ in \cite{Hess86} and has since turned into one of the most successful computational tools in relativistic quantum chemistry \cite{HRW02,RW04i,RW04ii,RW06i,RW06ii}.
From a mathematical point of view, the method of \cite{SS06} was extended to the multi-particle case in \cite{HubSto07} and to complex-dilated Dirac operators in \cite{Hub09}. Moreover, it was mentioned in \cite{Jak08point} that the same method works for non-vanishing magnetic fields ($\mathbf{A}\neq 0$).
Different techniques were employed in \cite{GNP89} and in \cite{Cor01, Cor04, Cor04ii, Cor07} to handle weaker electric potentials than the Coulomb potential.

By using a combination of the methods of \cite{LT} and \cite{SS06} we achieve generalizations of the aforementioned results; the main novelties are the following:

\begin{itemize}
\item When $\Phi=-\gm/|\cdot|$ is the Coulomb potential and $\mathbf{A}=0$, we obtain an exact block-diagonalization of $H$ up to nuclear charge $Z=124$, extending \cite[Theorem 1]{SS06}, where a transformation was shown to exist up to $Z=93$;
 
\item We show the convergence of the DKH approximation up to $Z=62$, extending \cite[Theorem~2]{SS06}, where the convergence was proved up to $Z=51$; 
 
\item Potentials are allowed to be (not necessarily symmetric) sesquilinear forms;

\item The method can be adapted to handle strong (e.g.\ constant, but also unbounded) magnetic fields.
\end{itemize}

Moreover, our transformation can be chosen as the direct rotation (see \cite{Da58}) between the subspace of upper (lower) component Dirac spinors and the positive (negative) spectral subspace of $H$, that is, it has minimal deviation from the identity among all such transformations. This is a consequence of the fact that the positive and negative spectral subspaces $Q_{\pm}\H$ of $H$ admit a representation in terms of so-called angular operators. Identifying the direct summands in the decomposition \eqref{natural decomposition} with the subspaces 
\[
P_{u}\H:=\SSet{\begin{pmatrix}u\\0\end{pmatrix}}{u\in L^2\left(\R^3,\C^2\right)},\quad P_{l}\H:=\SSet{\begin{pmatrix}0\\v\end{pmatrix}}{g\in L^2\left(\R^3,\C^2\right)}
\]
of ``upper-'' and ``lower-component" Dirac spinors, respectively, this means that
\begin{align*}
Q_+\H=\SSet{\begin{pmatrix}u\\X_+u\end{pmatrix}}{u\in P_{u}\H},\, Q_-\H=\SSet{\begin{pmatrix}X_-v\\v\end{pmatrix}}{v\in P_{l}\H},
\end{align*}
where $X_{\pm}$ are bounded operators in $L^2\left(\R^3,\C^2\right)$. In particular, we obtain bounds on the norms of $X_{\pm}$, which means that for any eigenfunction $\psi=(\Psi_{u},\Psi_{l})^{\rm t}$ corresponding e.g.\ to a positive eigenvalue of $H$, we must have
\[
\|\Psi_{l}\|\leq \|K_+\| \,\|\Psi_{u}\|.
\]

We emphasize that our technique is purely operator-theoretic in nature and thus not limited to the Dirac operator. Our main results, Theorems \ref{main theorem 1}, \ref{main theorem 1 sym.} and Theorems \ref{main theorem 2}, \ref{main theorem 2 sym.} apply to arbitrary self-adjoint operators with a spectral gap at zero, perturbed in the quadratic form sense.

We briefly sketch the outline of the paper. In Section 2, we set the necessary operator-theoretic background. Among the tools needed are indefinite quadratic form methods, spectral projections for non-selfadjoint unbounded operators and a theorem on accretive operators in spaces with an indefinite inner product. The main abstract results of the paper are stated in Section 3 and are applied to the Dirac operator on $\R^3$ with Coulomb-type potentials (with and without magnetic field) in Section 4. Section 5 contains the proofs of the main theorems.

\section{Preliminaries}

\subsection{Notation}
Let $X$ and $Y$ be Banach spaces. For an operator $S(X\to Y)$, we denote by $\dom(S)\subset X$ its domain and by $\Ran(S)\subset Y$ its range. All Banach spaces are always assumed to complex and operators between them are assumed to be linear. The Banach space of bounded operators from $X$ to $Y$ is denoted by $\LL(X,Y)$; if $X=Y$, we simply write $\LL(X):=\LL(X,X)$. For the identity operator in $X$ we write $I_X$ or $I$ if it is clear from the context which space is meant. If $Y=X$, then $\rho(S)$ and $\sigma(S)$ denote the resolvent set and spectrum of $S$, respectively. If $S$ is closed, then the former coincides with the set of all $z\in\C$ such that $S-z:\dom(S)\to X$ is bijective; here, we used the abbreviation $S-z:=S-z\,I$. If $S$ is closable, we denote its closure by $\overline{S}$. 
By an isomorphism between two Banach spaces we mean a linear homeomorhism. A subspace $\mathcal{L}\subset X$ is always understood to be closed. 
The topological direct sum of $X$ and $Y$ is denoted by $X\dotplus Y$. For two Hilbert spaces $\H$ and $\mathcal{K}$, the orthogonal sum is denoted by $\H\oplus\mathcal{K}$. Moreover, the scalar product $(\cdot,\cdot)$ in a Hilbert space $\H$ is assumed to be linear in the first variable, and for a densely defined operator $H$ in $\H$, its Hilbert space adjoint is denoted by $H^*$.
The Schatten-von Neumann ideals in $\H$ of order~$p$ are denoted by $\mathcal{S}_p(\H)$
For a sesquilinear form $\mathfrak{v}:\dom(\mathfrak{v})\times \dom(\mathfrak{v})\to \C$, the corresponding quadratic form is abbreviated by $\mathfrak{v}[u]:=\mathfrak{v}[u,u]$.
We say that an interval $(\al,\be)$ is a spectral gap for a self-adjoint operator $H$ if $(\al,\be)\subset\rho(H)$.
An integral $\int^{'}$ is always understood in the sense of the Cauchy principal value at zero and infinity.

\subsection{Indefinite quadratic forms}

An operator $H$ in a Hilbert space $\H$ is said to be \emph{associated} with a densely defined sesquilinear form $\mathfrak{h}$ if the following hold:
\begin{itemize}
 \item[\rm i)] $H$ is closed and densely defined;
 \item[\rm ii)] $\dom(H), \dom(H^*)\subset\dom(\mathfrak{h})$;
 \item[\rm iii)] $(Hu,v)=\mathfrak{h}[u,v]$, $u\in\dom(H)$, $v\in\dom(\mathfrak{h})$;
 \item[\rm iv)] $(u,H^*v)=\mathfrak{h}[u,v]$, $u\in\dom(\mathfrak{h})$, $v\in\dom(H^*)$.
\end{itemize}
If such an operator exists, then it is uniquely determined, see \cite[Proposition 2.3]{Veselic}. 
If $\mathfrak{h}$ is symmetric and $H$ is self-adjoint, then  $\mathfrak{h}$ is said to be \emph{represented} by $H$ if $\dom(\mathfrak{h})=\dom(|H|^{1/2})$ and
\[
\mathfrak{h}[u,v]=\left(|H|^{1/2}u,\operatorname{sign}(H)|H|^{1/2}v\right),\quad u,v\in \dom(\mathfrak{h}).
\]
The classical first representation theorem as found e.g.\ in \cite[Theorem VI.2.1.]{Ka} applies only to closed sectorial forms and will thus not be applicable in the present case. Instead, we shall borrow results from \cite{Veselic}, which generalize the well known pseudo-Friedrichs extension \cite[Theorem VI.3.11.]{Ka} to quadratic form perturbations.

\begin{hypothesis}\label{hypothesis form bounded}
Let $H_0$ be a self-adjoint operator in a Hilbert space $\H$, and let $\mathfrak{v}$ be a sesquilinear form on $\H$ such that 
for all $u,v\in \dom(\mathfrak{v})$
\begin{align}\label{rel. form bounded}
 |\mathfrak{v}[u,v]|\leq \|H_{a,b}^{1/2}u\|\|H_{a,b}^{1/2}v\|,\quad u,v\in \dom(\mathfrak{v}),\quad H_{a,b}:=a+b|H_0|,
\end{align}
for some $a,b\geq 0$, where $\dom(\mathfrak{v})$ is a core for $|H_0|^{1/2}$. 
\end{hypothesis}

Clearly, any sesquilinear form $\mathfrak{v}$ satisfying the assumptions of Hypothesis \ref{hypothesis form bounded} may be extended to a sesquilinear form on
$\mathcal{Q}:=\dom(|H_0|^{1/2})$ for which \eqref{rel. form bounded} continues to hold. We will therefore always assume that $\dom(\mathfrak{v})=\mathcal{Q}$.

\begin{remark}\label{remark eps 2 nonsymm.}
Condition \eqref{rel. form bounded} is equivalent to
\beqnt
 2|\mathfrak{v}[u,v]|\leq \left(a\|u\|^2+b\|\,H_0|^{1/2}u\|^2\right)+\left(a\|v\|^2+b\|\,H_0|^{1/2}v\|^2\right)
\eeqnt
If $\mathfrak{v}$ is a symmetric form, then \eqref{rel. form bounded} is equivalent to
\beq\label{v symmetric}
 |\mathfrak{v}[u]|\leq a\|u\|^2+b\|\,H_0|^{1/2}u\|^2, \quad u\in\dom(\mathfrak{v}).
\eeq
In general, \eqref{v symmetric} implies \eqref{rel. form bounded} with $a,b$ replaced by $2a,2b$.
\end{remark}

By the Riesz representation theorem, the formula
\beq\label{C bounded}
(C_{a,b}u,v)=\mathfrak{v}[H_{a,b}^{-1/2}u,H_{a,b}^{-1/2}v],\quad u,v\in\H,
\eeq
defines an operator $C_{a,b}\in\LL(\H)$ with $\|C_{a,b}\|\leq 1$. Assume that the operator-valued function $\widehat{C}_{a,b}:\rho(H_0)\to\C$,
\[
 \widehat{C}_{a,b}(z):=(H_0-z)H_{a,b}^{-1}+C_{a,b},\quad z\in\rho(H_0),
\]
has nonempty resolvent set, i.e.\ there exists $z_0\in\rho(H_0)$ such that $\widehat{C}_{a,b}(z_0)$ has a bounded inverse. By \cite[Theorem 2.4]{Veselic} there exists a unique operator $H$ associated to the quadratic form $\mathfrak{h}=\mathfrak{h}_0+\mathfrak{v}$, where
\begin{align*}
\mathfrak{h}_0[u,v]&:=\left(|H_0|^{1/2}u,\operatorname{sign}(H_0)|H_0|^{1/2}v\right),\quad u,v\in\mathcal{Q}
\end{align*}
is the form represented by $H_0$. More precisely, $H$ is given by the formulas
\begin{align}\label{H form construction}
 H-z&=H_{a,b}^{1/2}\widehat{C}_{a,b}(z)H_{a,b}^{1/2},\\
 H^*-\overline{z}&=H_{a,b}^{1/2}\widehat{C}_{a,b}(z)^*H_{a,b}^{1/2},\quad z\in\C,
\end{align}
and $\dom(H)$ is a core for $|H_0|^{1/2}$.
Whenever $\widehat{C}_{a,b}(z)$ is boundedly invertible, then $z\in\rho(H)$, $\overline{z}\in\rho(H^*)$, and
\begin{align}\label{resolvent H form construction}
 (H-z)^{-1}&=H_{a,b}^{-1/2}\widehat{C}_{a,b}(z)^{-1}H_{a,b}^{-1/2},\\
 (H^*-\overline{z})^{-1}&=H_{a,b}^{-1/2}\widehat{C}_{a,b}(z)^{-*}H_{a,b}^{-1/2}.
\end{align}
The construction does not depend on $a,b$. 

\begin{remark}\label{rem. v form of an operator}
If $\mathfrak{v}$ is the form of an operator $V$, 
\[
\mathfrak{v}[u,v]=(Vu,v),\quad u,v\in\dom(V),
\]
then $H$ is the pseudo-Friedrichs extension of $H_0+V$.
\end{remark}

We note that the construction of $H$ in \cite{Veselic} is accompanied by a spectral inclusion, see \cite[Theorems 2.11., 2.16., 3.1.]{Veselic}; compare also \cite{CueTre} for related results. For example, \cite[Theorems 3.1.]{Veselic} states that if $H_0$ has a spectral gap $(-\delta,\delta)$ and $a+b\,\delta<\delta$, then 
\[
(-\delta+a+b\,\delta,\delta-a-b\,\delta)+\I\R\subset \rho(H).
\]

\subsection{Spectral projections}

Let $S(X\to X)$ be an operator in a Banach space $X$, and let $Q_{\pm}\in\LL(X)$ be a pair of complementary projections, i.e.\ $Q_++Q_-=I$. Then $S$ is said to be \emph{decomposed} according to 
\beq\label{eq. decomp. of X}
X=Q_+X\dotplus Q_-X
\eeq
(compare \cite[III.5.6.]{Ka}) if
\beq\label{SQcommute}
Q_{\pm}\dom(S)\subset\dom(S),\quad SQ_{\pm}x=Q_{\pm}S x,\quad x\in\dom(S).
\eeq
With respect to the decomposition \eqref{eq. decomp. of X}, $S$ is then block-diagonal, 
\[
S=\begin{pmatrix}S_+&0\\0&S_-\end{pmatrix},
\]
where $S_{\pm}:=S|_{Q_{\pm}X}$ denote the parts of $S$ in $Q_{\pm}X$. Clearly, $\sigma(S)=\sigma(S_+)\cup\sigma(S_-)$, and $S_{\pm}$ are closed, densely defined etc.\ if and only if $S$ is. We are interested in the case where the union is disjoint; in particular, when $\I\R\subset\rho(S)$ and  
\beq\label{eq. separation of spectra}
\sigma(S_{\pm})=\sigma(S)\cap \C_{\pm},\quad \C_{\pm}:=\set{z\in\C}{\pm\,\re z>0}.
\eeq
If at least one of the the sets $\sigma(S)\cap\C_{\pm}$ is bounded, then $Q_{\pm}$ may be defined by the Riesz-Dunford functional calculus. If $S$ is a self-adjoint operator in a Hilbert space, then this may be accomplished by the self-adjoint functional calculus, even if both sets are unbounded. In either case, 
\beq\label{eq. diff. of proj}
\frac{1}{\pi\I}\int_{\I\R}^{'}(S-z)^{-1}\,\rd z x=Q_+x-Q_-x,\quad x\in X.
\eeq
In the general case, the problem of separating the spectrum at infinity arises. 
The following theorem was proved in \cite{LT} and is based on \cite[Theorem XV.3.1.]{GGK1}. We slightly simplify the assumptions stated in \cite[Theorem 1.1.]{LT}; they are equivalent by a straightforward Neumann series argument.

\begin{theorem}\label{LT GGK}
Let $S$ be a closed, densely defined operator in a Banach space $X$ such that $\I\R\subset\rho(S)$, $\lim_{|\eta|\to\infty}\|(S-\I\eta)^{-1}\|=0$ and
\[
\frac{1}{\pi\I}\int_{\I\R}^{'}(S-z)^{-1}\,\rd z 
\]
exists in the strong operator topology. Then there exist complementary projections $Q_{\pm}$ in $X$ such that $S$ is decomposed according to $X=Q_+X\dotplus Q_-X$ and such that \eqref{eq. separation of spectra}--\eqref{eq. diff. of proj} hold.
\end{theorem}

\begin{remark}
We call $Q_{\pm}$ the \emph{spectral projections} corresponding to the right and left half planes $\C_{\pm}$. We remark that $S$ is bisectorial under the stated conditions and that the spectral projections may in principle also be defined by the functional calculus for such operators, see e.g.\ \cite{McIntosh86}. However, the two notions need not coincide; in particular, the spectral projections defined by the functional calculus may be unbounded.
\end{remark}

We will need the following perturbation result:

\begin{theorem}\label{th. integral of resolvent difference exists in norm}
Assume Hypothesis \rref{hypothesis form bounded} and that \eqref{rel. form bounded} holds with $a,b\geq 0$, $b<1$. Then there exists a unique closed densely defined operator $H$ which is associated to the quadratic form $\mathfrak{h}=\mathfrak{h}_0+\mathfrak{v}$; moreover, $\dom(H)$ is a core for $|H_0|^{1/2}$.
If $\,\I\R\subset\rho(H_0)\cap\rho(H)$, then the assertions of Theorem \rref{LT GGK} hold for $H$.
\end{theorem}

\begin{proof}
The first part was proved in \cite[Theorem 2.11]{Veselic}.

\medskip
To show that $H$ satisfies the assumptions of of Theorem \rref{LT GGK}, we prove that the integral
\beqnt
\int_{-\infty}^{\infty '}\left((H-\I\eta)^{-1}-(H_0-\I\eta)^{-1}\right)\,\rd \eta,
\eeqnt
exists in the norm operator topology and that
\beq\label{eq. resolvent decay}
\sup_{\eta\in\R}|\eta|\left\|(H-\I\eta)^{-1}-(H_0-\I\eta)^{-1}\right\|<\infty.
\eeq

Since the assumptions of Theorem \ref{LT GGK} are obviously satisfied for the self-adjoint operator $H_0$, it then follows that the same holds true for $H$, by virtue of the identity
\[
(H-\I\eta)^{-1}=(H_0-\I\eta)^{-1}+\left((H-\I\eta)^{-1}-(H_0-\I\eta)^{-1}\right).
\]

\medskip
By the spectral theorem for self-adjoint operators 
\beq\label{eq. CH0Hableqsup}
\|(H_0-\I\eta)^{-1}H_{a,b}\|\leq\sup_{|t|\geq\delta}\frac{a+b|t|}{\sqrt{t^2+\eta^2}},\quad \eta\in\R.
\eeq
A straightforward computation yields for the supremum above (see e.g.\ \cite{Veselic})
\beq\label{eq. supremum}
 \sup_{|t|\geq\delta}\frac{a+b|t|}{\sqrt{t^2+\eta^2}}=\frac{1}{|\eta|}\sqrt{a^2+b^2\eta^2},\quad \eta\in\R.
\eeq
Since $b<1$, it follows from \eqref{eq. CH0Hableqsup} and \eqref{eq. supremum} that for each $\widetilde{b}\in(b,1)$ there exists $R>0$ such that 
\beq\label{eq. eq. CH0Hableqbprime}
\|(H_0-\I\eta)^{-1}H_{a,b}\|\leq \widetilde{b},\quad \eta\in\R,\,|\eta|\geq R
\eeq
We may assume without loss of generality that \eqref{eq. eq. CH0Hableqbprime} holds for all $\eta\in\R$. Otherwise, the existence in norm of the two integrals

\beqnt
\int_{-R}^{R}\left((H-\I\eta)^{-1}-(H_0-\I\eta)^{-1}\right)\,\rd \eta, \quad  \int_{|\eta|\geq R}^{'}\left((H-\I\eta)^{-1}-(H_0-\I\eta)^{-1}\right)\,\rd \eta.
\eeqnt
is shown separately. But since the integrand is continuous as function of $\eta\in\R$, the first integral above always exists, while for the second, \eqref{eq. eq. CH0Hableqbprime} holds.

By \eqref{resolvent H form construction}, we have for all $\eta\in\R$
\beqnt\begin{split}
 (H-\I\eta)^{-1}&=H_{a,b}^{-1/2}\left((H_0-\I\eta)H_{a,b}^{-1}+C_{a,b}\right)^{-1}H_{a,b}^{-1/2}\\ 
&=(H_0-\I\eta)^{-1}H_{a,b}^{1/2}\left(1+C_{a,b}(H_0-\I\eta)^{-1}H_{a,b}\right)^{-1}H_{a,b}^{-1/2}\\
&=(H_0-\I\eta)^{-1}H_{a,b}^{1/2}\sum_{n=0}^{\infty}\left[-C_{a,b}(H_0-\I\eta)^{-1}H_{a,b}\right]^{n} H_{a,b}^{-1/2}\\[-2mm]
&=(H_0-\I\eta)^{-1}+(H_0-\I\eta)^{-1}H_{a,b}^{1/2}\sum_{n=1}^{\infty}\left[-C_{a,b}(H_0-\I\eta)^{-1}H_{a,b}\right]^{n} H_{a,b}^{-1/2}.
\end{split}\eeqnt
The sum above converges absolutely by \eqref{eq. eq. CH0Hableqbprime} and because $\|C_{a,b}\|\leq 1$. Hence,
\begin{align}\label{eq. diff. of resolvents with Dn}
 (H-\I\eta)^{-1}-(H_0-\I\eta)^{-1}&=(H_0-\I\eta)^{-1}H_{a,b}^{1/2}\sum_{n=1}^{\infty}\left[-C_{a,b}(H_0-\I\eta)^{-1}H_{a,b}\right]^{n} H_{a,b}^{-1/2}\notag\\
&=(H_0-\I\eta)^{-1}H_{a,b}^{1/2}\sum_{n=1}^{\infty} D_n(\I\eta) (H_0-\I\eta)^{-1} H_{a,b}^{1/2},
\end{align}
where
\[
 D_n(z):=\left[-C_{a,b}(H_0-z)^{-1}H_{a,b}\right]^n (H_0-z)H_{a,b}^{-1},\quad z\in\rho(H_0).
\]
Note that $D_n(z)$ contains $n$ factors of $C_{a,b}$ and $n-1$ factors of $(H_0-z)^{-1}H_{a,b}$. Therefore, by \eqref{eq. eq. CH0Hableqbprime},
\[
 \|D_n(z)\|\leq \|C_{a,b}\|^n\|(H_0-z)^{-1}H_{a,b}\|^{n-1}\leq \widetilde{b}^{n-1}.
\]
We set 
\[
 G_{a,b}(z):=(H_0-z)^{-1}H_{a,b}^{1/2}\in\LL(\H),\quad z\in\rho(H_0).
\]
Then, for $u,v\in\H$,
\begin{align}\label{normZAZB}
\left|\left(\left[(H-\I\eta)^{-1}-(H_0-\I\eta)^{-1}\right]u,v\right)\right|
&=\left|\left(\sum_{n=1}^{\infty}D_n(\I\eta) G_{a,b}(\I\eta)u,G_{a,b}(\I\eta)v\right)\right|\notag\\
&\leq(1-\widetilde{b})^{-1} \|G_{a,b}(\I\eta)u\|\,\|G_{a,b}(\I\eta)v\|.
\end{align}

Let $(\rho_n)_{n\in\N}\subset(0,\infty)$ be such that $\rho_n\to\infty$, and define $\{T_n\}_{n\in\N}$ by
\[
 T_n:=\frac{1}{\pi\I}\int_{-\rho_n}^{\rho_n}\left((H-\I\eta)^{-1}-(H_0-\I\eta)^{-1}\right)\,\rd \eta.
\]
Note that, since the integrand is norm-continuous, the integral exists in norm.  
By~\eqref{normZAZB} and the Cauchy-Schwarz inequality, we have for all $u,v\in\H$, 
\[\begin{split}
&\int_{-\rho_n}^{\rho_n}\left|\left(\left[(H-\I\eta)^{-1}-(H_0-\I\eta)^{-1}\right]u,v\right)\right|\,\rd \eta\\
&\leq(1-\widetilde{b})^{-1}\int_{-\rho_n}^{\rho_n}\|G_{a,b}(\I\eta)u\|\,\|G_{a,b}(\I\eta)v\|\,\rd \eta\\
&\leq(1-\widetilde{b})^{-1}\left(\int_{-\rho_n}^{\rho_n}\|G_{a,b}(\I\eta)u\|^2\,\rd \eta\right)^{1/2}
\left(\int_{-\rho_n}^{\rho_n}\|G_{a,b}(\I\eta)v\|^2\,\rd \eta\right)^{1/2}\\[2mm]
&=\pi(1-\widetilde{b})^{-1}\left(a\,\|H_0^{-1}\|+b\right)\,\|u\|\,\|v\|,
\end{split}\]
where the last equality is a consequence of the spectral theorem. Therefore, $\{(T_n u,v)\}_{n\in\N}$ converges uniformly for $u,v$ in the unit ball of $\H$. By \cite[p.~150]{Ka}, it follows that $\{T_n\}_{n\in\N}$ converges in norm.

Another application of the spectral theorem yields, using \eqref{normZAZB},
\[\begin{split}
 &\|(H-\I\eta)^{-1}-(H_0-\I\eta)^{-1}\|=\sup_{\|u\|=\|v\|=1}\left|\left(\left[(H-\I\eta)^{-1}-(H_0-\I\eta)^{-1}\right]u,v\right)\right|\\
&\leq(1-\widetilde{b})^{-1}\|G_{a,b}(\I\eta)\|^2=(1-\widetilde{b})^{-1}\sup_{t\in\sigma(H_0)}\frac{a+b|t|}{|\eta|+t^2|\eta|^{-1}}\,\frac{1}{|\eta|}\\
&\leq(1-\widetilde{b})^{-1}\sup_{t\in\sigma(H_0)}\frac{a+b|t|}{2|t|}\,\frac{1}{|\eta|}\leq(1-\widetilde{b})^{-1}\left(\frac{a}{2}\,\|H_0^{-1}\|+b\right)\,\frac{1}{|\eta|}.
\end{split}\]
This proves \eqref{eq. resolvent decay}.
\end{proof}

\subsection{Graph subspaces and angular operators}

\begin{definition}
Let $X,Y$ be Banach spaces and $Z=X\dotplus Y$. A subspace $\mathcal{L}\subset Z$ is called a \emph{graph subspace} with respect to $X$ if there exists an operator $A_X\in\LL(X,Y)$ such that
\[
\mathcal{L}=\set{x+A_X x}{x\in X}.
\]
In this case, $A_X$ is called the \emph{angular operator} of $\mathcal{L}$ with respect to $X$. 
\end{definition}

For simplicity, we shall also call a subspace $\mathcal{M}\subset Y$ of the form
\[
\mathcal{M}=\set{A_Y y+y}{y\in Y},
\]
$A_Y\in\LL(Y,X)$, a graph subspace, although the term ´´inverse graph subspace" would be more appropriate. 

\begin{remark}
Let $P_X$ be the projection of $Z$ onto $X$ along $Y$ and $P_Y=I_Z-X$. It is easy to see that $\mathcal{L}\subset Z$ is a graph subspace with respect to $X$ if and only if
\[
P_X|_{\mathcal{L}}:\mathcal{L}\to X
\]
is an isomorphism and that the angular operator $A_X$ is given by
\[
A_X=P_Y(P_X|_{\mathcal{L}})^{-1}.
\]
\end{remark}

\begin{proposition}\label{proposition inverse W}
Let $X,Y$ be Banach spaces, $Z=X\dotplus Y$, and let $\mathcal{L}, \mathcal{M}\subset Z$ be graph subspaces with respect to $X$ and $Y$, with angular operators $A_X\in\LL(X,Y)$ and $A_Y\in\LL(Y,X)$, respectively. Then the following are equivalent.
\begin{itemize}
\item[\rm i)] $I_X-A_Y A_X$ has a bounded inverse;
\item[\rm ii)] $I_Y-A_X A_Y$ has a bounded inverse;
\item[\rm iii)] The operator
\[
W:=\begin{pmatrix}I_X&A_Y\\A_X&I_Y\end{pmatrix}\in \LL(Z)
\]
has a bounded inverse, with
\beq\label{eq. Winv.}
W^{-1}=\begin{pmatrix}(I_X-A_Y A_X)^{-1}&-(I_X-A_Y A_X)^{-1}A_Y\\-(I_Y-A_X A_Y)^{-1}A_X&(I_Y-A_X A_Y)^{-1}\end{pmatrix}\in \LL(Z)
\eeq
\item[\rm iv)] $\mathcal{L}\dotplus \mathcal{M}=Z$.
\end{itemize}
\end{proposition}

\begin{proof}
The equivalence of i)-iii) follows from the Schur-Frobenius factorization; for example
\[
\begin{pmatrix}I_X&A_Y\\A_X&I_Y\end{pmatrix}=\begin{pmatrix}I_X&0\\A_X&I_Y\end{pmatrix}\begin{pmatrix}I_X&0\\0&I_Y-AX A_Y\end{pmatrix}\begin{pmatrix}I_X&A_Y\\0&I_Y\end{pmatrix}.
\]
The formula \eqref{eq. Winv.} is easily verified by a direct computation.
To prove the equivalence of iii) and iv), we observe that since
\[
 W\begin{pmatrix}x\\y\end{pmatrix}=\begin{pmatrix}x\\A_X x\end{pmatrix}+\begin{pmatrix}A_Y y\\y\end{pmatrix},
\]
$W$ is surjective if and only if $\mathcal{L}+\mathcal{M}=\H$, and $W$ is injective if and only if $\mathcal{L}\cap \mathcal{M}=\emptyset$. An application of the closed graph theorem completes the proof.
\end{proof}

In the Hilbert space situation, we have the following useful proposition see e.g.\ \cite[Corollary 3.4]{KMM03}.

\begin{proposition}\label{graph subspace-pair of projections}
Let $\H$ be a Hilbert space and $\mathcal{L},\mathcal{M}\subset\H$ subspaces with corresponding orthogonal projections $P_{\mathcal{L}}$, $P_{\mathcal{M}}$. Then $\|P_{\mathcal{L}}-P_{\mathcal{M}}\|<1$ if and only if $\mathcal{M}$ is the graph of an operator $K\in\LL(\mathcal{L},\mathcal{L}^\perp)$.
In this case
\beq
\|K\|=\frac{\|P_{\mathcal{L}}-P_{\mathcal{M}}\|}{\sqrt{1-\|P_{\mathcal{L}}-P_{\mathcal{M}}\|^2}},\quad \|P_{\mathcal{L}}-P_{\mathcal{M}}\|=\frac{\|K\|}{\sqrt{1+\|K\|^2}}.
\eeq
\end{proposition}

It is easy to see that $d(\mathcal{L},\mathcal{M}):=\|P_{\mathcal{L}}-P_{\mathcal{M}}\|$ defines a metric on the set of subspaces of a Hilbert space $\H$. Let us introduce the \emph{angular metric} by
\[
d_{a}(\mathcal{L},\mathcal{M}):=\arcsin \|P_{\mathcal{L}}-P_{\mathcal{M}}\|.
\] 
The fact that $d_{a}$ is indeed a metric (i.e.\ satisfies the triangle inequality) was proven in \cite{Brown93}. The angular metric is related to the operator angle $$\Theta_{\mathcal{L},\mathcal{M}}:=\arcsin\sqrt{I_{\mathcal{L}}-P_{\mathcal{L}}P_{\mathcal{M}}}$$ between the subspaces $\mathcal{L}$ and $\mathcal{M}$ by the equality $d_{a}(\mathcal{L},\mathcal{M})=\|\Theta_{\mathcal{L},\mathcal{M}}\|$.

\subsection{Accretive operators in indefinite inner product spaces}

Our main tool in the proof of Theorems \ref{main theorem 1} and \ref{main theorem 1 sym.} is the following result about accretive operators in indefinite inner product spaces. It is a slight generalization of \cite[Theorem 1.4.]{LT}. 

\begin{definition}
 Let $\H$ be a Hilbert space and $W$ a bounded self-adjoint operator on $\H$. 
\begin{itemize}
 \item[\rm i)] An operator $T(\H\to\H)$ is called $W$-\emph{accretive} if
\[
 \re\, (WTx,x)\geq 0\quad x\in\dom(T).
\]
 \item[\rm ii)] A linear manifold $\mathcal{L}\subset\H$ is called $W$-\emph{nonnegative} {\rm(}$W$-\emph{nonpositive}{\rm)} if 
\[
 (Wx,x)\geq 0\,\,\,{\rm(}\leq 0{\rm)},\quad x\in \mathcal{L}.
\]
\end{itemize}
\end{definition}

\begin{theorem}\label{W-space accretive thm.}
Let $\H$ be a Hilbert space and $W$ a bounded self-adjoint operator on $\H$. Let $T(\H\to\H)$ be a closed, $W$-accretive operator such that $\I\R\setminus\{0\}\subset\rho(T)$. Assume that the integral
\beq\label{eq. integral over resolvent in W-space accretive thm}
\frac{1}{\pi\I}\int_{\I\R}^{'}(T-z)^{-1}\rd z,
\eeq
exists in the weak operator topology and is the difference of two complementary projections $Q_{\pm}~\in~\LL(\H)$,
\beq\label{eq. integral is the difference of projections}
\frac{1}{\pi\I}\int_{\I\R}^{'}(T-z)^{-1}\rd z=Q_+-Q_-.
\eeq
Then $Q_+\H\subset\H$ is a $W$-nonnegative subspace and $Q_-\H\subset\H$ is a $W$-nonpositive subspace.

If $0\in\rho(W)$, we denote by $P_{\pm}$ the spectral projections onto the positive and negative spectral subspace of $W$. Set $\H_{\pm}:=P_{\pm}\H$ and $W_{\pm}:=\pm W|_{\H_{\pm}}$. Then there exist operators $K_{\pm}\in\LL(\H{\pm},\H_{\mp})$ such that 
\beq\label{Graph subspaces}
Q_+\H=\set{x+K_+ x}{x\in P_+\H}, \quad Q_-\H=\set{y+K_- y}{y\in P_-\H}.
\eeq
Moreover,
\[
 \|K_{\pm}\|\leq\sqrt{\frac{\sup\sigma(W_{\pm})}{\inf\sigma(W_{\mp})}}.
\]
\end{theorem}

\begin{proof}
Let $x\in Q_+\H$. Then
\[\begin{split}
 [x,x]&=\re [x,x]=\re[(Q_+-Q_-)x,x]\\
&=\re\left(\frac{1}{\pi\I}\int_{\I\R}^{'}[(T-z)^{-1}x,x]\rd z\right)\\
&=\frac{1}{\pi}\int_{\R}^{'}\re\, [T(T-\I\eta)^{-1}x,(T-\eta)^{-1}x]\rd \eta\geq 0.
\end{split}\]
Thus, $Q_+\H$ is nonnegative. Analogously, one shows that $Q_-\H$ is nonpositive. 

\medskip
If $0\in\rho(W)$, then $\H$ equipped with the indefinite inner product $(W\cdot,\cdot)$ is a Krein space with fundamental decomposition $\H=\H_+[+]\H_-$. Since $Q_+ + Q_-=I$, \cite[I.1.25]{AI} implies that the subspace $Q_{+}\H$ is maximal nonnegative and $Q_{+}\H$ is maximal nonpositive. It follows that there exist operators $K_{\pm}\in\LL(\H{\pm},\H_{\mp})$ such that \eqref{Graph subspaces} holds, see e.g.\ \cite[Theorem II.11.7]{Bog74} or \cite{Lan82}. Moreover, $K_{\pm}$ are contractions with respect to the canonical norms $\|W_{\pm}^{1/2}\cdot\|$ in $P_{\pm}\H$ induced by $W$. Stated differently,
\[
 \|W{\mp}^{1/2}K_{\pm}W{\pm}^{-1/2}\|\leq 1,
\]
which implies
\[
 \|K_{\pm}\|\leq \|W_{\mp}^{-1/2}\|\,\|W_{\pm}^{1/2}\|=\sqrt{\frac{\sup\sigma(W_{\pm})}{\inf\sigma(W_{\mp})}}.
\]
The last equality is a consequence of the spectral theorem for self-adjoint operators.
\end{proof}

More advanced results of this kind, stated in terms of interpolation spaces, may be found in \cite{Pyatkov10}.

\section{Main results}

\begin{theorem}\label{main theorem 1}
Let $H_0$ be a self-adjoint operator with a spectral gap $(-\delta,\delta)$ in a Hilbert space $\H$. Let $\mathfrak{v}$ be a quadratic form such that $\dom(\mathfrak{v})$ is a core for $|H_0|^{1/2}$ and such that
\begin{align*}
\operatorname*{\sup_{x,y\in\H}}_{\|x\|=\|y\|=1}&\left|\mathfrak{v}\left[|H_0|^{-1/2}x,|H_0|^{-1/2}y\right]\right|<1,\\
\operatorname*{\sup_{x,y\in P_+\H\cup P_-\H}}_{\|x\|=\|y\|=1}&\left|\mathfrak{v}\left[|H_0|^{-1/2}x,|H_0|^{-1/2}y\right]\right|=:\rho <\frac{1}{2}.
\end{align*}

Furthermore, let $\widetilde{P}_{\pm}$ be a pair of complementary projections in $\H$ such that 
\[
\nu:=\|P_{\pm}-\widetilde{P}_{\pm}\|<1,\quad \arctan \sqrt{\frac{\rho}{2-3\rho}}+\arcsin \nu <\frac{\pi}{2}.
\]

Then the following hold:
\begin{itemize}
\item[\rm i)] There exists a unique closed densely defined operator $H$
associated to the quadratic form $\mathfrak{h}:=\mathfrak{h}_0+\mathfrak{v}$, and $\dom(H)$ is a core for $|H_0|^{1/2}$. 

\item[\rm ii)] There exist complementary projections $Q_{\pm}$ in $\H$ such that $H$ is decomposed according to $\H=Q_+H\dotplus Q_-H$, and 
\[
\sigma(H|_{Q_+\H})=\sigma(H)\cap\C_{\pm}.
\]
\item[\rm iii)] The restrictions $\widetilde{P}_{\pm}|_{Q_{\pm}\H}:Q_{\pm}\H\to\widetilde{P}_{\pm}\H$ are isomorphisms, and, with
 $X_{\pm}:=\widetilde{P}_{\mp}(\widetilde{P}_{\pm}|_{Q_{\pm}\H})^{-1}$, we have
\[
Q_+\H=\set{x+X_+x}{x\in \widetilde{P}_+\H},\quad Q_-\H=\set{y+X_-y}{y\in \widetilde{P}_-\H},
\]
\[
\|X_{\pm}\|\leq \tan\left(\arctan \sqrt{\frac{\rho}{2-3\rho}}+\arcsin \nu\right).
\]
\end{itemize}
\end{theorem} 

\begin{remark}\label{rem. Momega}
It is easy to see that for any invariant set $M\subset\H$ of $H_0$, the following are equivalent:
\begin{itemize}
\item[\rm i)] There exists $\omega\geq 0$ such that
\[
\operatorname*{\sup_{x,y\in M}}_{\|x\|=\|y\|=1}\left|\mathfrak{v}\left[|H_0|^{-1/2}x,|H_0|^{-1/2}y\right]\right|<\omega;
\]
\item[\rm ii)] There exist $a,b\geq 0$, $a+b\,\delta<\omega\delta$, such that \eqref{rel. form bounded} holds on $M$.
\end{itemize} 
\end{remark}

\begin{theorem}\label{main theorem 1 sym.}
Let $H_0$ be a self-adjoint operator with a spectral gap $(-\delta,\delta)$ in a Hilbert space $\H$. Let $\mathfrak{v}$ be a symmetric quadratic form such that $\dom(\mathfrak{v})$ is a core for $|H_0|^{1/2}$ and such that
\begin{align*}
\operatorname*{\sup_{x,y\in\H}}_{\|x\|=\|y\|=1}&\left|\mathfrak{v}\left[|H_0|^{-1/2}x,|H_0|^{-1/2}y\right]\right|<\infty,\\
\operatorname*{\sup_{x,y\in P_+\H\cup P_-\H}}_{\|x\|=\|y\|=1}&\left|\mathfrak{v}\left[|H_0|^{-1/2}x,|H_0|^{-1/2}y\right]\right|=:\rho <1.
\end{align*}

Furthermore, let $\widetilde{P}_{\pm}$ be a pair of complementary projections in $\H$ such that 
\[
\nu:=\|P_{\pm}-\widetilde{P}_{\pm}\|<1,\quad \arctan \sqrt{\frac{\rho}{2-\rho}}+\arcsin \nu<\frac{\pi}{2}.
\]

Then the following hold:
\begin{itemize}
\item[\rm i)] There exists a unique self-adjoint operator $H$
associated to the quadratic form $\mathfrak{h}:=\mathfrak{h}_0+\mathfrak{v}$; moreover, $0\in\rho(H)$, and $\dom(H)$ is a core for $|H_0|^{1/2}$. 

\item[\rm ii)] Let $Q_{\pm}$ denote the spectral projections of $H$ corresponding to the positive and negative spectrum, respectively. Then $\widetilde{P}_{\pm}|_{Q_{\pm}\H}:Q_{\pm}\H\to\widetilde{P}_{\pm}\H$ are isomorphisms, and, with
 $X_{\pm}:=\widetilde{P}_{\mp}(\widetilde{P}_{\pm}|_{Q_{\pm}\H})^{-1}$, we have
\[
Q_+\H=\set{x+X_+x}{x\in \widetilde{P}_+\H},\quad Q_-\H=\set{y+X_-y}{y\in \widetilde{P}_-\H},
\]
\[
\|X_{\pm}\|\leq \tan\left(\arctan \sqrt{\frac{\rho}{2-\rho}}+\arcsin\nu\right),\quad \quad X_-=-X_+^*.
\]
\end{itemize}
\end{theorem}

\begin{remark}
If $\widetilde{P}_{\pm}$ are orthogonal projections, then the result of Theorem \ref{main theorem 1 sym.} may be equivalently stated as
\[
\|\widetilde{P}_{\pm}-Q_{\pm}\|\leq \sin\left(\arcsin \sqrt{\frac{\rho}{2}}+\arcsin \nu\right).
\]
\end{remark}

An immediate consequence of Theorems \ref{main theorem 1} and \ref{main theorem 1 sym.} is a block-diagonalization of $H$.

\begin{corollary}\label{cor. diagonalization Dirac}
With respect to the decomposition $\H=\widetilde{P}_+\H\dotplus\widetilde{P}_-\H$, we have
\begin{align}\label{block-diagonalization}
\begin{pmatrix}I&X_-\\X_+&I\end{pmatrix}^{-1}H\begin{pmatrix}I&X_-\\X_+&I\end{pmatrix}=\begin{pmatrix}Z_+&0\\0&Z_-\end{pmatrix},
\end{align}
where $Z_{\pm}$ are similar to $H|_{Q_{\pm}\H}$. 

If $\nu=0$ and $(H-\I\eta)^{-1}-(H_0-\I\eta)^{-1}\in\mathcal{S}_p(\H)$ for some {\rm(}and hence for all{\rm)} $\eta\in\R$, then $X_{\pm}\in\mathcal{S}_p(\H)$.  
\end{corollary}

\begin{proof}
We set
\[
W:=\begin{pmatrix}I&X_-\\X_+&I\end{pmatrix}.
\]
Since $Q_{\pm}\H$ are complementary graph subspaces with angular operators $X_{\pm}$, it is easily seen that $Q_{\pm}W=W\widetilde{P}_{\pm}$ and that $W$ is bijective and hence boundedly invertible by the closed graph theorem. 
Since $H$ is decomposed according to $Q_+\H\dotplus Q_-\H$, it follows that $W^{-1}HW$ is decomposed according to $\widetilde{P}_+\H\dotplus \widetilde{P}_-\H$, i.e.\ it is block-diagonal with respect to this decomposition, and \eqref{block-diagonalization}
holds for some closed densely defined operators $Z_{\pm}$ on $\widetilde{P}_{\pm}\H$. Noting that 
\[
W|_{\widetilde{P}_{\pm}\H}=(\widetilde{P}_{\pm}|_{Q_{\pm}\H})^{-1},
\]
we find that 
\[
Z_{\pm}=(\widetilde{P}_{\pm}|_{Q_{\pm}\H})H|_{Q_{\pm}\H}(\widetilde{P}_{\pm}|_{Q_{\pm}\H})^{-1}.
\]

\medskip
If the resolvent difference of $H_0$ and $H$ belongs to $\mathcal{S}_p(\H)$, then 
\[
Q_{\pm}-\widetilde{P}_{\pm}=\nlim_{\rho\to\infty} \frac{\pm 1}{2\pi}\int_{-\rho}^{\rho}(H-\I\eta)^{-1}-(H_0-\I\eta)^{-1}\,\rd\eta\in \mathcal{S}_p(\H).
\]
With respect to the decomposition $\H=\widetilde{P}_+\H\dotplus \widetilde{P}_-\H$, e.g.\ the projections $\widetilde{P}_+$ and~$Q_+$ are given by
\beq\label{eq. Qplus}
\widetilde{P}_+=\begin{pmatrix}I&0\\0&0\end{pmatrix},\quad  Q_+=\begin{pmatrix}(I-X_-X_+)^{-1}&-(I-X_- X_+)^{-1}X_-\\X_+(I-X_- X_+)^{-1}&-X_+(I-X_- X_+)^{-1}X_-\end{pmatrix},
\eeq
see \cite[p.\ 63]{AI}. It follows that
\[
X_+(I-X_-X_+)^{-1}\in \mathcal{S}_p(\H),
\]
and hence
\[
X_+=X_+(I-X_-X_+)^{-1}(I-X_-X_+)\in \mathcal{S}_p(\H) \qedhere.
\]
\end{proof}

\begin{remark}
In the case when $\mathfrak{v}$ is symmetric, the fact that $X_-=X_+^*$ implies that the bounded operators $1-X_{\mp}X_{\pm}$ are self-adjoint and uniformly positive. The transformation $W$ in Corollary \ref{cor. diagonalization Dirac}
can then easily be made unitary by multiplying from the right with $\diag(\Omega_+,\Omega_-)$,
\[
\Omega_{\pm}:=\left(1-X_{\mp}X_{\pm}\right)^{-1/2},
\]
see \cite{LT}. In fact, the resulting operator takes the more familiar form (on the right-hand side)
\beq\label{Uunitary}
\begin{pmatrix}\Omega_+&X_-\Omega_-\\X_+\Omega_+&\Omega_-\end{pmatrix}=[I-(Q_+-P_+)^2]^{-1/2}[P_+Q_++P_-Q_-],
\eeq
as can be checked by a straightforward computation, using Phroposition \ref{graph subspace-pair of projections} and \eqref{eq. Qplus}. This is the direct rotation between the subspaces $P_{\pm}\H$ and $Q_{\pm}\H$, see \cite{Da58}.
\end{remark}

For the next two theorems, we restrict ourselves to the case $\widetilde{P}_{\pm}=P_{\pm}$. 
The aim is to approximate the exact block-diagonal operator by simpler operators which are amendable to computation. To this end we introduce a coupling constant $\gm$ for the perturbation $\mathfrak{v}$ and consider the family of operators 
\beqnt
H(\gm):=H_0+\gm \mathfrak{v},\quad\gm\in \mathds{D},
\eeqnt
where $\mathds{D}\subset\C$ is the open unit disk. Notice that Theorem \ref{main theorem 1} i) is valid for the whole family $H(\gm)$, $\gm\in\mathds{D}$, while ii) only holds for $\gm\in(-1,1)$ since $\gm \mathfrak{v}$ is not symmetric for $\gm\in\C\setminus\R$.
Therefore, under the assumptions of Theorem \ref{main theorem 1} i), the following operators are well-defined for any $\gm\in\mathds{D}$:

\begin{itemize}
\item The spectral projections $Q_{\pm}(\gm)$ of $H(\gm)$ corresponding to $\C_{\pm}$; 
\item The angular operators $X_{\pm}(\gm):=P_{\mp}\left(P_{\pm}|_{Q_{\pm}(\gm)}\right)^{-1}$; 
\item The inverse square roots $\Omega_{\pm}(\gm):=\left(1-X_{\mp}(\gm)X_{\pm}(\gm)\right)^{-1/2}$.
\end{itemize}

The latter can no longer be defined by the spectral theorem. Rather, since $\sup_{\gm\in\mathds{D}}\|X_{\pm}(\gm)\|<1$, we may define them by a norm-convergent power series
\beq\label{inverse square root series for Omega}
 \Omega_{\pm}(\gm)=\sum_{n=0}^{\infty}\begin{pmatrix}-1/2\\n\end{pmatrix}\left(-X_{\mp}(\gm)X_{\pm}(\gm)\right)^n.
\eeq

Let $H_{\rm diag}(\gm)$ be the block-diagonal operator (with respect to the decomposition $\H=P_+\H\oplus P_- \H$)
\[
H_{\rm diag}(\gm):=W(\gm)^{-1}H(\gm)W(\gm),\quad W(\gm):=\begin{pmatrix}I&X_-(\gm)\\X_+(\gm)&I\end{pmatrix}.
\]
Moreover, if $\mathfrak{v}$ is symmetric, define $\widehat{H}_{\rm diag}(\gm)$ by
\begin{align*}
\widehat{H}_{\rm diag}(\gm):=U(\gm)^{-1}H(\gm)U(\gm),
\quad U(\gm):=\begin{pmatrix}\Omega_+(\gm)&X_-(\gm)\Omega_-(\gm)\\X_+(\gm)\Omega_+(\gm)&\Omega_-(\gm)\end{pmatrix}.
\end{align*}
We show that the operators $H^N_{\rm diag}(\gm)$, $\widehat{H}^N_{\rm diag}(\gm)$ corresponding to the formal $N$-th order Taylor polynomials of $H_{\rm diag}(\gm)$, respectively $\widehat{H}_{\rm diag}(\gm)$, converge in the norm-resolvent sense to the exact block-diagonal operators as $N\to\infty$. Furthermore, we make precise in what sense these operators may be viewed as Taylor polynomials.

\medskip
For technical reasons we assume that $\mathfrak{v}$ is the form of an operator $V$, i.e.\
\[
\mathfrak{v}[u,v]=(V u,v),\quad u,v\in\dom(\mathfrak{v})=\dom(V),
\]

\begin{theorem}\label{main theorem 2}
Assume that $\mathfrak{v}$ in Theorem \rref{main theorem 1} is the form of an $H_0$-bounded operator $V$ with $H_0$-bound less than $\sqrt{(2-3\rho)/(2-2\rho)}$,
then $\{H_{\rm diag}(\gm)\}_{\gm\in\mathds{D}}$ is a holomorphic family of type {\rm (}A{\rm )} with $\dom(H_{\rm diag}(\gm))=\dom(H_0)$. For $\gm\in\mathds{D}$ and $N$ large enough, the operators $H_{\rm diag}^N(\gm)$, defined by
\[
H_{\rm diag}^N(\gm)u:=\sum_{n=0}^N \frac{\gm^n}{n!}\frac{\rd^n}{\rd\gm^n}\left(H_{\rm diag}(\gm)u\right)|_{\gm=0},\quad u\in\dom(H_0),
\]
have nonempty resolvent set and $H_{\rm diag}^N(\gm)\to H_{\rm diag}(\gm)$ as $N\to\infty$ in the norm-resolvent topology, uniformly on compact subsets of $\mathds{D}$.
\end{theorem}

\begin{theorem}\label{main theorem 2 sym.}
Assume that $\mathfrak{v}$ in Theorem \rref{main theorem 1} is the form of an $H_0$-bounded symmetric operator $V$ such that $\|VH_0^{-1}\|<1$. Then $\{\widehat{H}_{\rm diag}(\gm)\}_{\gm\in\mathds{D}}$ is a self-adjoint holomorphic family, and 
\[
T(\gm):=\overline{|H_0|^{-1/2}\widehat{H}_{\rm diag}(\gm)|H_0|^{-1/2}},\quad \gm\in\mathds{D},
\]
is bounded-holomorphic. For $\gm\in(-1,1)$ and $N$ large enough, there exists a unique self-adjoint operator $\widehat{H}_{\rm diag}^N(\gm)$ associated to the quadratic form
\[
\widehat{\mathfrak{h}}_{\rm diag}^N(\gm)[u,v]:=\sum_{n=0}^N \frac{\gm^n}{n!}\frac{\rd^n}{\rd\gm^n}\left.\left(T(\gm)|H_0|^{1/2}u,|H_0|^{1/2}v\right)\right|_{\gm=0},\quad u,v\in \mathcal{Q},
\]
and $\dom(\widehat{H}_{\rm diag}^N(\gm))$ is a core for $|H_0|^{1/2}$. Moreover, $\widehat{H}_{\rm diag}^N(\gm)\to \widehat{H}_{\rm diag}(\gm)$ as $N\to\infty$ in the norm-resolvent topology, uniformly on compact subsets of $(-1,1)$.
\end{theorem}

\begin{remark}
The norm-resolvent convergence in Theorem \ref{main theorem 2} implies convergence of the spectra and spectral projections of $\widehat{H}_{\rm diag}^N(\gm)$ to those of $\widehat{H}_{\rm diag}(\gm)$ since all operators are self-adjoint, see \cite[VIII.7.]{RS1}. 

On the other hand, in Theorem \ref{main theorem 1}, norm-resolvent convergence only implies that if $z\in \rho(H_{\rm diag}(\gm))$, then $z\in \rho(H_{\rm diag}^N(\gm))$ for $N$ sufficiently large. By contraposition, this means that if $z\in\sigma(H_{\rm diag}^N(\gm))$ for all $N$ sufficiently large, then $z\in\sigma(H_{\rm diag}(\gm))$. However, it can easily be seen from the proofs of Theorems \ref{th. integral of resolvent difference exists in norm} and \ref{main theorem 2} that the spectral projections corresponding to $\C_{\pm}$ do converge.    
\end{remark}

\begin{remark}
Like in \cite{SS06}, the method to prove convergence is based on analyticity. 
We remark that even if $\mathfrak{v}$ is symmetric, the detour through non self-adjoint operators (i.e.\ complex $\gm$) is unavoidable. Indeed, in order to estimate the radius of convergence of the polynomial approximation, we need complex analyticity, not just real analyticity.
\end{remark}


\section{Application to Dirac operators}

\subsection{The free Dirac operator}

Let us denote the free (i.e\ $\Phi=0$, $\mathbf{A}=0$) Dirac operator defined on the first order Sobolev space by $H_0$, i.e.\
\beq\label{Dirac operator}
H_0=\begin{pmatrix}1&-\I\,\boldsymbol{\sigma}\cdot\nabla\\-\I\,\boldsymbol{\sigma}\cdot\nabla&-1\end{pmatrix}, \quad \dom(H_0)=H^1(\R^3,\C^2)\oplus H^1(\R^3,\C^2).
\eeq
It is well-known that that the spectrum of $H_0$ is absolutely continuous and equal to
\beq\label{spec H_0}
 \sigma(H_0)=(-\infty,-1]\cup[1,\infty).
\eeq

The Foldy-Wouthuysen transformation $U_{\mathrm{FW}}$, in momentum space a multiplication operator by the function
\beq
u(\mathbf{p}):=\frac{1+E(p)+\be \boldsymbol{\al}\cdot \mathbf{p}}{N(p}=\frac{1}{N(p)}\begin{pmatrix}1+E(p)&\boldsymbol{\sigma}\cdot\mathbf{p}\\-\boldsymbol{\sigma}\cdot\mathbf{p}&1+E(p)\end{pmatrix},
\eeq
where 
\[
E(p):=\sqrt{1+p^2},\quad N(p):=\sqrt{2E(p)(1+E(p))},\quad p:=|\mathbf{p}|,\quad \mathbf{p}\in\R^3,
\]
block-diagonalizes the free Dirac operator: 
\[
 U_{\mathrm{FW}}\begin{pmatrix}I&-\I\boldsymbol{\sigma}\cdot\nabla\\-\I\boldsymbol{\sigma}\cdot\nabla&-I\end{pmatrix}U_{\mathrm{FW}}^{*}=\begin{pmatrix}\sqrt{1-\Delta}&0\\0&-\sqrt{1-\Delta}\end{pmatrix}.
\]
Moreover, the orthogonal projections $P_{\pm}=1/2(1 \pm H_0|H_0|^{-1})$ onto the positive and negative spectral subspaces of $H_0$ in momentum space are given by
\beq\label{free projections Dirac}
 \Lambda_{\pm}(\mathbf{p})
=\frac{1}{2E(p)}\begin{pmatrix}\pm 1+E(p)&\boldsymbol{\sigma}\cdot\mathbf{p}\\ \boldsymbol{\sigma}\cdot\mathbf{p}&\mp 1+E(p)\end{pmatrix}.
\eeq
For fixed $\mathbf{p}\in\R^3$, the matrices $\Lambda_{\pm}(\mathbf{p})$ are orthogonal projections in $\C^4$, and both possess a two-dimensional eigenspace corresponding to the eigenvalue $1$. For example, given two orthonormal vectors $u_1,u_2\in\C^2$, the vectors
\[
 \psi_i:=\frac{1}{N(p)}\begin{pmatrix}(1+E(p))u_i\\\boldsymbol{\sigma}\cdot\mathbf{p} \,u_i\end{pmatrix},\quad i=1,2,
\]
form an orthonormal basis of the eigenspace belonging to $\Lambda_{+}(\mathbf{p})$. It is thus seen that 
\[
P_+\H=\set{\mathcal{F}^{-1}\frac{1}{N(\cdot)}\begin{pmatrix}(1+E(\cdot))\widehat{u}(\cdot)\\ \boldsymbol{\sigma}\cdot(\cdot)\widehat{u}(\cdot)\end{pmatrix}}{u\in\H_{\rm u}}
=\set{\begin{pmatrix}u\\X_+ u\end{pmatrix}}{u\in\H_{\rm u}},
\]
where $\mathcal{F}$ is the Fourier transformation, $\widehat{u}:=\mathcal{F} u$ and 
\beqnt
X_+:=\mathcal{F}^{-1}\frac{\boldsymbol{\sigma}\cdot(\cdot)}{1+E(\cdot)}\mathcal{F}=-\I\cdot\boldsymbol{\sigma}\cdot\nabla\left(1+\sqrt{1-\Delta}\right)^{-1}\in\LL(\H_{\rm u},\H_{\rm l})
\eeqnt
is the angular operator of $P_+\H$ with respect to $\H_{\rm u}$.

\subsection{Assumptions on the potential}

We impose the following general conditions on the potential $V$.

\begin{hypothesis}\label{hypothesis general matrix potential}
Let $V$ be a measurable $4\times 4$-matrix-valued function, and assume that there exist $a,b\geq 0$ such that
\beq\label{assumption on V c}
 \|V(x)\|_{\LL(\C^4)}\leq a+\frac{b}{|x|}, \quad \mbox{for almost all } x\in\R^3.
\eeq
\end{hypothesis}

In most physically relevant applications, $V$ is given by
\beq\label{V electromagnetic}
V=\begin{pmatrix}\Phi &-\boldsymbol{\sigma}\cdot\mathbf{A}\\-\boldsymbol{\sigma}\cdot\mathbf{A}&\Phi \end{pmatrix}.
\eeq
However, we will not assume $V$ to be of that particular form. Most importantly, our assumptions cover the Coulomb potential $V=Z\al/|\cdot|$, where $\al\approx 1/137$ is the fine structure constant.

\medskip
The following proposition summarizes some useful inequalities. 

\begin{proposition}\label{prop. inequalities free Dirac}
Let $H_0$ be the free Dirac operator \eqref{Dirac operator}. Then the following hold.
\begin{itemize}
\item[\rm i)] $\displaystyle\|\,|\cdot|^{-1}u\|\leq 2 \ \, \|H_0 u\|$ for all $u\in\dom(H_0)$;\\
\item[\rm ii)] $\displaystyle|(|\cdot|^{-1}u,u)|\leq \pi/2 \, \|\,|H_0|^{1/2} u\|^2$ for all $u\in\mathcal{Q}$;\\
\item[\rm iii)] $\displaystyle|(|\cdot|^{-1}u,u)|\leq 1/2\cdot(\pi/2+2/\pi) \, \|\,|H_0|^{1/2} u\|^2$ for all $u \in P_+\mathcal{Q}\cup P_-\mathcal{Q}$.
\end{itemize}
\end{proposition}
Inequalities i) and ii) are the classical Hardy, respectively Kato inequalities. Inequality iii) is due to \cite{Tix1, Tix2} and \cite{BE}, see also \cite{BE11}. It establishes the boundedness from below of the Brown-Ravenhall operator and was also used in \cite{DES2000} to prove variational principles for Dirac operators with Coulomb-like potentials.

\subsection{Magnetic fields}

If we want to include magnetic fields (i.e.\ $\mathbf{A}\neq 0$ in \eqref{V electromagnetic}), then the condition \eqref{assumption on V c} is too restrictive since it does not allow e.g.\ for constant fields $\mathbf{B}=\operatorname{curl} \mathbf{A}$. The magnetic field term should thus not be considered as a perturbation; rather, we should include it in the definition of the unperturbed operator.
If, for example, ${\mathbf A}\in L^3_{{\rm loc}}(\R^3,\R^3)$, then the magnetic Dirac operator
\beq\label{magnetic Dirac operator}
H_{{\mathbf A}}=\begin{pmatrix}1&-\I\boldsymbol{\sigma}\cdot\nabla_{{\mathbf A}}\\-\I\boldsymbol{\sigma}\cdot\nabla_{{\mathbf A}}&-1\end{pmatrix}, 
\eeq
is essentially self-adjoint on $C_0^{\infty}(\R^3\setminus\left\{0\right\})^4$, see e.g.\ \cite[notes sect.\ 4.3]{Th}. 
In the following, we denote its self-adjoint closure by the same symbol $H_{{\mathbf A}}$. Owing to the special structure of the operator matrix, $(-1,1)$ is a spectral gap for $H_{{\mathbf A}}$.
By combining Hardy's and Kato's inequalities with the diamagnetic inequality for Schr\"odinger operators \cite[Theorem 4.5.1.]{BE11}, one can prove the following proposition, see e.g.\ \cite[(4.7), (4.9)]{Jak06}.

\begin{proposition}\label{prop. inequalities magnetic Dirac}
Assume that ${\mathbf A}\in L^3_{{\rm loc}}(\R^3,\R^3)$ and that $\mathbf{B}:=\operatorname{curl}\mathbf{A}$ is essentially bounded.
Then the following hold.

\begin{itemize}
 \item[\rm i)] $\displaystyle\|\,|\cdot|^{-1}u\|^2\leq 4\,\|H_{\mathbf{A}} u\|^2+4\cdot\|\mathbf{B}\|_{\infty}\,\|u\|^2$ for all $u\in\dom(H_A)$;\\

 \item[\rm ii)] $\displaystyle |(|\cdot|^{-1}u,u)|\leq \pi/2\,\|\,|H_{\mathbf{A}}|^{1/2}u\|^2+\pi/2\cdot\sqrt{\|\mathbf{B}\|_{\infty}}\,\|u\|^2$
for all $u\in\dom(|H_A|^{1/2})$.
\end{itemize}

\end{proposition}

As was noted in \cite{Jak08heat}, the boundedness assumption on the ${\mathbf B}$-field may be relaxed by using an estimate due to Balinsky, Evans and Lewis \cite{BEL01}, relating the Schr\"odinger operator $(\nabla-\I\mathbf{A})^2$ to the Pauli operator $(\boldsymbol{\sigma}\cdot\nabla_{{\mathbf A}})^2$, when the latter has no zero modes. 

\begin{proposition}\label{prop. ineq. with magn. field unbounded B}
Assume that ${\mathbf A}\in L^3_{{\rm loc}}(\R^3,\R^3)$, and let $\mathbf{B}:=\operatorname{curl}\mathbf{A}$, 
\beq\label{eq. deltamB}
\delta({\mathbf B}):=\inf_{x\in\H,\,\|x\|=1}\|(1-S^*S)x\|,\quad S:=|{\mathbf B}|^{1/2}\left((\sigma\cdot\nabla_{{\bf A}})^2+|{\mathbf B}|\right)^{-1/2}.
\eeq
Then $0<\delta({\mathbf B})\leq 1$, and the following hold.

\begin{itemize}
 \item[\rm i)] $\displaystyle\|\,|\cdot|^{-1}u\|\leq 2\cdot \delta({\mathbf B})^{-1}\, \|H_{\mathbf A} u\|^2$ for all $u\in\dom(H_{\mathbf A})$;\\

 \item[\rm ii)] $\displaystyle|(\,|\cdot|^{-1}u,u)|\leq \pi/2\cdot \delta({\mathbf B})^{-1}\,\|H_{\mathbf A} u\|^2$ for all $u\in\dom(|H_{\mathbf A}|^{1/2})$.

\end{itemize}
\end{proposition}

\subsection{Block-diagonalization and convergence of the DKH method}

For notational convenience, we subsequently identify the abstract unperturbed operator $H_0$ in Section 3 either with the free Dirac operator (denoted by the same symbol above) or with the magnetic Dirac  operator $H_{\mathbf A}$, depending on whether $\mathbf A$ vanishes or not. Correspondingly, the projections $P_{\pm}$ in Section 3 are identified with the spectral projections of the free or the magnetic Dirac operator. 

\begin{proposition}\label{prop. 1oversqert2}
We have $\|P_{u}-P_+\|=\|P_{l}-P_-\|=1/\sqrt{2}$.
\end{proposition}

\begin{proof}
For the free Dirac operator, this can easily be inferred from \eqref{free projections Dirac}. For the magnetic Dirac operator, it follows from supersymmetry arguments, see \cite{LT}. 
\end{proof}

In view of Remark \ref{rem. Momega} and Propositions \ref{prop. inequalities free Dirac}-\ref{prop. ineq. with magn. field unbounded B}, sufficient conditions for Theorems \ref{main theorem 1}, \ref{main theorem 1 sym.}, \ref{main theorem 2}, \ref{main theorem 2 sym.} to hold may easily be obtained in terms of $a$ and $b$. For example, by Proposition \ref{prop. 1oversqert2}, in the absence of magnetic fields ($\mathbf{A}=0$), the assumptions of Theorem \ref{main theorem 1 sym.} are satisfied for $\widetilde{P}_+=P_{u}$, $\widetilde{P}_-=P_{l}$ (i.e.\ for the natural decomposition \eqref{natural decomposition}) whenever 
\[
a+\frac{b}{2}\left(\frac{\pi}{2}+\frac{2}{\pi}\right)<1,
\]
whereas for Theorem \ref{main theorem 2 sym.} the left hand side has to be less than $1/2$. 
In particular, for the Coulomb potential $V=Z\al/|\cdot|$ the above inequalities amount to $Z\leq 124$ and $Z\leq 62$, respectively. 
In the case $\mathbf{A}\neq 0$, the upper bound for $Z$ depends on the magnetic field; Proposition \ref{prop. ineq. with magn. field unbounded B} yields $Z\cdot\delta(\mathbf{B})^{-1}\leq 87$ (as compared to $Z\cdot\delta(\mathbf{B})^{-1}\leq 60$ in \cite{Jak08heat}).

\section{Proofs of the main results}

\begin{proof}[\rm{\textbf{Proof of Theorem \ref{main theorem 1}}}]
i) follows from \cite[Theorem 3.1]{Veselic} and Remark \ref{rem. Momega}; in addition, we have $\I\R\subset\rho(H)$.

\medskip
ii) is then a direct consequence of Theorem \ref{th. integral of resolvent difference exists in norm}. 

\medskip
iii) is shown in two steps. First for $\nu=0$ and then for the general case.

We claim that $H$ is $W_i$-accretive for $i=1,2$ and
\[
 W_1:=\frac{\rho}{2-3\rho}\,P_+-P_-,\quad  W_2:=P_+-\frac{\rho}{2-3\rho}\,P_-.
\]
Theorem~ \rref{W-space accretive thm.} then implies the existence of the angular operators $X^0_{\pm}$ of $Q_{\pm}\H$ with respect to $P_{\pm}\H$ such that
\[
\|X^0_{\pm}\|\leq \sqrt{\frac{\rho}{2-3\rho}}.
\]
This completes the proof in the case $\nu=0$.

For the general case, let $\widehat{P}_{\pm}$ and $\widehat{Q}_{\pm}$ be the orthogonal projections onto $\widetilde{P}_{\pm}\H$ and $Q_{\pm}\H$, respectively. By \cite[Theorem I.6.35.]{Ka}, we have
\[
\|\widehat{P}_{\pm}-P_{\pm}\|\leq \|\widetilde{P}_{\pm}-P_{\pm}\|,\quad \|\widehat{Q}_{\pm}-P_{\pm}\|\leq \|Q_{\pm}-P_{\pm}\|. 
\]
Hence, by the triangle inequality for the angular metric and Proposition \ref{graph subspace-pair of projections},
\begin{align*}
\arcsin \|\widehat{P}_{\pm}-\widehat{Q}_{\pm}\|&\leq  \arcsin \|P_{\pm}-\widehat{Q}_{\pm}\|+\arcsin \|\widehat{P}_{\pm}-P_{\pm}\|\\
&\leq  \arcsin \|P_{\pm}-Q_{\pm}\|+\arcsin \|\widetilde{P}_{\pm}-P_{\pm}\|\\
&\leq \arcsin \frac{\|X^0_{\pm}\|}{\sqrt{1+\|X^0_{\pm}\|^2}}+\arcsin \nu= \arctan \|X^0_{\pm}\|+\arcsin \nu\\
&\leq \arctan \sqrt{\frac{\rho}{2-3\rho}}+\arcsin \nu <\pi/2,
\end{align*}
which is equivalent to 
\[
\|\widehat{P}_{\pm}-\widehat{Q}_{\pm}\|\leq \sin\left(\arctan \sqrt{\frac{\rho}{2-3\rho}}+\arcsin \nu\right)<1.
\]
By Proposition \ref{graph subspace-pair of projections} again, $\widehat{Q}_{\pm}\H=Q_{\pm}\H$ is a graph subspace with respect to $\widehat{P}_{\pm}\H=\widetilde{P}_{\pm}\H$, with angular operators $X_{\pm}$ satisfying
\[
\|X_{\pm}\|=\frac{\|\widehat{P}_{\pm}-\widehat{Q}_{\pm}\|}{\sqrt{1-\|\widehat{P}_{\pm}-\widehat{Q}_{\pm}\|^2}}\leq \tan\left(\arctan \sqrt{\frac{\rho}{2-3\rho}}+\arcsin\nu\right).
\]

\medskip
We now show that $H$ is $W_i$-accretive for $i=1,2$. To this end, for $\mu_{\pm}>0$, we set $W:=\mu_+P_+-\mu_-P_-$. Since $P_{\pm}\mathcal{Q}\subset\mathcal{Q}$, the following sesquilinear forms are well-defined on $\mathcal{Q}$: 
\begin{align*}
 \hat{\mathfrak{v}}_{\rm D}[u,v]&:=\mu_+ \, \mathfrak{v}[P_+u,P_+v]-\mu_- \, \mathfrak{v}[P_-u,P_-v],\\
 \hat{\mathfrak{v}}_{\rm O}[u,v]&:=\mu_+ \, \mathfrak{v}[P_-u,P_+v]-\mu_- \, \mathfrak{v}[P_+u,P_-v],\\
 \chi_{\pm}[u,v]&:=\mu_{\pm} \, \|\,H_0|^{1/2}P_{\pm}u\| \, \|\,H_0|^{1/2}P_{\pm}v\|+\mu_{\pm} \, \re\,\mathfrak{v}_{\rm D}[P_{\pm}u,P_{\pm}v].
\end{align*}
Then, for $u\in\dom(H)$,
\beq\label{reHuWu}\begin{split}
&\re\,(WHu,u)=\re\,(Hu,Wu)=\re\,\mathfrak{h}[u,Wu]=\re\,\mathfrak{h}_{0}[u,Wu]+\re\,\mathfrak{v}[u,Wu]\\
&=\re\,\mathfrak{h}_0[u,Wu]+\re\,\hat{\mathfrak{v}}_{\rm D}[u]+\re\,\hat{\mathfrak{v}}_{\rm O}[u]
=\chi_+[u]+\chi_-[u]+\re\,\hat{\mathfrak{v}}_{\rm O}[u].
\end{split}\eeq
By assumption, we have
\begin{align*}
|\re\,\hat{\mathfrak{v}}_{\rm D}[u]|&\leq\rho\cdot\mu_+\,\|\,H_0^{1/2}P_+u\|^2+\rho\cdot\mu_-\,\|\,H_0^{1/2}P_-u\|^2,\\
|\re\,\hat{\mathfrak{v}}_{\rm O}[u]|&\leq\sigma|\re\,\mathfrak{v}[P_-u,P_+u]|
\leq\rho\cdot\frac{\sigma}{2}\,\|\,H_0^{1/2}P_+u\|^2+\rho\cdot\frac{\sigma}{2}\,\|\,H_0^{1/2}P_-u\|^2,
\end{align*}
where 
\[
\sigma:=\mu_+ +\mu_-.
\]
It follows that
\[
 \chi_{\pm}[u]\geq\mu_{\pm}(1-\rho)\|\,|H_0|^{1/2}P_{\pm}u\|^2\geq 0,
\]
and, by \eqref{reHuWu},
\[
\re\,(WHu,u)\geq\left(1-\frac{\rho\cdot\sigma}{2\,\mu_+\cdot(1-\rho)}\right)\chi_+[u]
+\left(1-\frac{\rho\cdot\sigma}{2\,\mu_-\cdot(1-\rho)}\right)\chi_-[u].
\]
For $\mu_+=\rho/(2-3\rho)$ and $\mu_-=1$, the expressions in brackets are nonnegative. The same is true for $\mu_+=1$ and $\mu_-=\rho/(2-3\rho)$.
\end{proof}

\begin{proof}[\rm{\textbf{Proof of Theorem \ref{main theorem 1 sym.}}}]
i) follows from \cite[Theorem 2.16]{Veselic} and Remark \ref{rem. Momega}.

\medskip
ii) The proof is identical to the one of Theorem \ref{main theorem 1} iii) above; the only difference is that here we can choose
\[
\sigma:=|\mu_+-\mu_-|,
\] 
and $H$ will be $W_i$-accretive for
\[
 W_1:=\frac{\rho}{2-\rho}\,P_+-P_-,\quad  W_2:=P_+-\frac{\rho}{2-\rho}\,P_-.
\]
Since $H$ is self-adjoint, the equality $X_-=X_+^*$ follows from the orthogonality of $Q_+\H$ and $Q_-\H$.
\end{proof}


Before we proceed to the proof of Theorem \ref{main theorem 2}, let us briefly recall the following proposition, which we shall subsequently use without further mention. The proof is a straightforward application of the Neumann series, see e.g.\ \cite[Lemma 1.8.1]{Ya}.

\begin{proposition}\label{lemma inverse holomorphic}
 Let $X$ be a Banach space and $\Omega\subset\C$ open. Assume that the operator-valued function $T(\cdot):\Omega\to X$ is holomorphic in $z_0\in\Omega$, and that $T(z_0)$ has a bounded inverse. Then $T(z)$ has a bounded inverse in a neighbourhood of~$z_0$, and $T(\cdot)^{-1}$ is holomorphic in $z_0$.
\end{proposition}

\begin{proof}[\rm{\textbf{Proof of Theorem \ref{main theorem 2}}}]
We first show that $W(\gm)$ is holomorphic in $\mathds{D}$.

Replacing $C$ by $\gm C$ in \eqref{eq. diff. of resolvents with Dn}, one observes that $$(H(\gm)-\I\eta)^{-1}-(H_0-\I\eta)^{-1}$$
is a norm-convergent power series for $\gm\in\mathds{D}$. By the proof of Lemma \ref{th. integral of resolvent difference exists in norm} the integral
\[
 Q_{\pm}(\gm)-P_{\pm}=\frac{1}{2\pi\I}\int_{\R}^{'}\left((H(\gm)-z)^{-1}-(H_0-z)^{-1}\right)\,\rd z
\]
converges in the norm operator topology, uniformly in $\gm$ on compact subsets of $\mathds{D}$. This shows that $\{H(\gm)\}_{\gm\in\mathds{D}}$ and $\{Q_{\pm}(\gm)\}_{\gm\in\mathds{D}}$ are holomorphic families. 

From \eqref{eq. Qplus} we infer that 
\[
 X_+(\gm)=Q_{21}(\gm)Q_{11}(\gm)^{-1},\quad\ \gm\in\mathds{D},\quad Q_+(\gm)=:\left(Q_{ij}(\gm)\right)_{i,j=1}^2,
\]
whence $\{X_{+}(\gm)\}_{\gm\in\mathds{D}}$ is holomorphic. The proof for $\{X_{-}(\gm)\}_{\gm\in\mathds{D}}$ is analogous. Consequently, $\{W(\gm)\}_{\gm\in\mathds{D}}$ is holomorphic. 

\medskip
Since the $H_0$-bound of $V$ is less than one, the family $\{H(\gm)\}_{\gm\in\mathds{D}}$ is holomorphic of type (A), with $\dom(H(\gm))=\dom(H_0)$ for all $\gm\in\mathds{D}$ \cite[VII]{Ka}. We now show that $H_{\rm diag}(\gm)$ is also holomorphic of type (A). It is sufficient to show that
\beq\label{eq. WdomH0subsetdomH0}
W(\gm)\dom(H_0)\subset\dom(H_0),\quad \gm\in\mathds{D}.
\eeq
Indeed, since $H_0$ is closed and $W(\gm)$ is bounded, the operator $$Y(\gm):=H_0W(\gm)H_0^{-1}$$ is closed and thus bounded by the closed graph theorem. It follows that 
\beq\label{eq. domain constant}
\dom(H_{\rm diag}(\gm))=\dom(H_0W(\gm))=\dom(Y(\gm)H_0)=\dom(H_0),\quad \gm\in\mathds{D}.
\eeq

It remains to show that $H_{\rm diag}(\gm)u$ is holomorphic for every $u\in\dom(H_0)$. Since $H(\gm)H_0^{-1}$ is bounded and holomorphic in norm, the same applies to $H_{\rm diag}(\gm)H_0^{-1}$ in virtue of
\[
H_{\rm diag}(\gm)H_0^{-1}=W(\gm)^{-1}H(\gm)W(\gm)H_0^{-1}=W(\gm)^{-1}H(\gm)H_0^{-1}Y(\gm).
\]
In particular, $H_{\rm diag}(\gm)u=H_{\rm diag}(\gm)H_0^{-1}(H_0u)$ is holomorphic for  every $u\in\dom(H_0)$.
%
%
%
%

\medskip
We now prove \eqref{eq. WdomH0subsetdomH0}; note that this is equivalent to
\beq\label{eq. XdomH0subsetdomH0}
X_{\pm}(\gm)P_{\pm}\dom(H_0)\subset P_{\mp}\dom(H_0),\quad \gm\in\mathds{D}.
\eeq
Since $P_{\pm}\dom(H_0)\subset\dom(H_0)$, we can write $H(\gm)$ as an operator matrix with respect to the decomposition $\H=P_+\H\oplus P_-\H$ as follows:
\[
 H(\gm)=\begin{pmatrix}
         A&B\\C&D\end{pmatrix}:=\begin{pmatrix}
         P_+(H_0+\gm V)P_+&P_+(H_0+\gm V)P_-\\P_-(H_0+\gm V)P_+&P_-(H_0+\gm V)P_-\end{pmatrix}.
\]
By the Schur-Frobenius factorization (see e.g.\ \cite{CT} for unbounded operators), the bounded invertibility of $(H(\gm)-\I\eta)$ is equivalent to the bounded invertibility of e.g.\ the first Schur complement
\[
S_+(\I\eta):=A-\I\eta-B(D-\I\eta)^{-1}C,\quad \dom(S_+(\I\eta)):=P_+\dom(H_0).
\]
Moreover, we have
\beq\label{eq. inverse Schur}
(H(\gm)-\I\eta)^{-1}=\begin{pmatrix}S_+(\I\eta)^{-1}&-S_+(\I\eta)^{-1}B(D-\I\eta)^{-1}\\ *&*\end{pmatrix},
\eeq
where the lower entries of the matrix are bounded operators which can be expressed in terms of $B$, $C$, $D$ and $S_+(\I\eta)^{-1}$; we won't need the explicit expressions here.

Since $Q_+H(\gm)\subset H(\gm)Q_+$, it follows that $Q_+(H(\gm)-\I\eta)^{-1}=(H(\gm)-\I\eta)^{-1}Q_+$. Comparing the ranges of the operators on the left and right hand side yields
\[
Q_+\H\cap\dom(H_0)=(H(\gm)-\I\eta)^{-1}Q_+\H.
\]
Projecting onto $P_+\H$ on either side, we obtain, using \eqref{eq. inverse Schur},
\beq\label{eq. set Schur}
\set{x\in P_+\H}{X_+(\gm)x\in P_-\H}=\set{S_+(\I\eta)^{-1}(1-B(D-\I\eta)^{-1})X_+(\gm)x}{x\in P_+\H}.
\eeq
Setting
\begin{align*}
p&:=\|P_+VP_-(P_-H_0P_--\I\eta)^{-1}\|,\\
q&:=\|P_-VP_-(P_-H_0P_--\I\eta)^{-1}\|,\\
r&:=\|V(H_0-\I\eta)^{-1}\|,
\end{align*}
and observing that $p^2+q^2\leq r^2$, we obtain, by a Neumann series argument,
\[
\|B(D-\I\eta)^{-1}\|\leq \frac{p}{1-q}\leq \frac{r}{\sqrt{1-r^2}}.
\]
Denoting by $b_0$ the $H_0$-bound of $V$, we find that
\[
\lim_{|\eta|\to\infty} \|B(D-\I\eta)^{-1}X_+(\gm)\|^2\leq \frac{b_0^2}{1-b_0^2}\cdot\frac{\rho}{2-3\rho}<1.
\]
Hence, for $|\eta|$ sufficiently large, $1-B(D-\I\eta)^{-1}X_+(\gm)$ is a isomorphism in $P_+\H$, and the right hand side of \eqref{eq. set Schur} equals $P_+\dom(H_0)$. Therefore, we have
\[
\set{x\in P_+\H}{X_+(\gm)x\in P_-\H}=P_+\H,
\]
which is equivalent to the first inclusion in \eqref{eq. XdomH0subsetdomH0}. The second inclusion is shown analogously, by using the second Schur complement
\[
S_-(\I\eta):=D-\I\eta-C(A-\I\eta)^{-1}B,\quad \dom(S_-(\I\eta)):=P_-\dom(H_0).
\]

\medskip
The fact that $H_{\rm diag}(\gm)$ is holomorphic of type (A) now implies the norm-resolvent convergence of its Taylor series. We notice that the proof given here can be adapted to the case when $H_0$ is not boundedly invertible (or even when $\H$ is only a Banach space) by regarding the following operators as maps from the Banach space $\dom(H_0)$ (with the graph norm) into $\H$.

For $N\in\N$ and $\gm\in\mathds{D}$, define the operators $R^N(\gm)$ by
\[
R^N(\gm)u:=H_{\rm diag}(\gm)u-H_{\rm diag}^N(\gm)u,\quad u\in\dom(R^N(\gm)):=\dom(H_0).
\]
Since $H_{\rm diag}(\gm)H_0^{-1}$ is bounded-holomorphic in $\mathds{D}$, its Taylor series converges uniformly on every compact subset  $K\subset\mathds{D}$, which means that
\[
\sup_{\gm\in K}\|R^N(\gm)H_0^{-1}\|\to 0,\quad N\to\infty.
\]
By the stability of bounded invertibility \cite[Theorem IV.1.16.]{Ka} it thus follows that $H_{\rm diag}^N(\gm)$
has a bounded inverse for sufficiently large $N$; moreover, by the second resolvent identity,
\[\begin{split}
\|H_{\rm diag}(\gm)^{-1}-H_{\rm diag}^N(\gm)^{-1}\|&\leq \|H_{\rm diag}^N(\gm)^{-1}\|\,\|R^N(\gm)H_{\rm diag}(\gm)^{-1}\|\\
&\leq \frac{\|H_{\rm diag}(\gm)^{-1}\|\,\|R^N(\gm)H_{\rm diag}(\gm)^{-1}\|}{1-\|R^N(\gm)H_{\rm diag}(\gm)^{-1}\|}.
\end{split}\]
The latter converges to zero as $N\to \infty$ since
\[\begin{split}
\|R^N(\gm)H_{\rm diag}(\gm)^{-1}\|
&=\|R^N(\gm)H_0^{-1}Y(\gm)H_0H(\gm)^{-1}W(\gm)^{-1}\|\\
&\leq\|R^N(\gm)H_0^{-1}\|\,\|Y(\gm)\|\,\|H_0H(\gm)^{-1}\|\,\|W(\gm)^{-1}\|.
\end{split}\]
The convergence is uniform in $\gm\in K$ since the functions $\|H_{\rm diag}^N(\cdot)^{-1}\|$, $\|Y(\cdot)\|$ etc.\ are continuous from $K$ to $[0,\infty)$ and hence take their maximum on the compact set $K$.
\end{proof}









\begin{proof}[\rm{\textbf{Proof of Theorem \ref{main theorem 2 sym.}}}]
Like in the proof of Theorem \ref{main theorem 2}, one shows that $K_{\pm}(\gm)$ are holomorphic in $\mathds{D}$. The absolute convergence of the series \ref{inverse square root series for Omega}, locally uniformly in $\gm\in\mathds{D}$, implies the holomorphy of $\Omega_{\pm}(\gm)$ and hence of $U(\gm)$. The equation
\beqnt
  \widehat{H}_{\rm diag}(\gm)^{-1}=U(\gm)^{-1}H(\gm)^{-1}U(\gm),\quad  \gm\in\mathds{D},
\eeqnt
then shows that $\widehat{H}_{\rm diag}(\gm)$ is holomorphic in $\mathds{D}$. 

Clearly, since $V$ is symmetric,
%
%
$\{H(\gm)\}_{\gm\in\mathds{D}}$ is a self-adjoint family, i.e.\
\[
H(\gm)=H^*(\overline{\gm}),\quad \gm\in\mathds{D}.
\]
By the identity theorem for bounded-holomorphic functions, we have  
\beqnt
X_-(\gm)=-X_+(\overline{\gm})^*,\quad U(\gm)U(\overline{\gm})^*=U(\overline{\gm})^*U(\gm)=I
\eeqnt
for all $\gm\in\mathds{D}$. It follows that $\widehat{H}_{\rm diag}(\gm)$ is a self-adjoint family.

\medskip
We now show that for all $\gm\in\mathds{D}$, $|H_0|^{-1/2}\widehat{H}(\gm)|H_0|^{-1/2}$ extends to a bounded operator by closure.
We claim that it is sufficient to show the following:
\beq\label{eq. X mapping property}
X_{\pm}(\gm)P_{\pm}\mathcal{Q}\subset \mathcal{Q},\quad \|\,|H_0|^{1/2}X_{\pm}(\gm)|H_0|^{-1/2}\|<1.
\eeq

Indeed, it then follows that for all $u\in P_{\pm}\H$, $v\in P_{\pm}\mathcal{Q}$, $\|u\|=\|v\|=1$,
\[\begin{split}
&\left|(\Omega_{\pm}(\gm)|H_0|^{-1/2}u,|H_0|^{1/2}v)\right|\\
=&\left|\sum_{n=0}^{\infty}\begin{pmatrix}-1/2\\n\end{pmatrix}\left(\left[-X_{\mp}(\gm)X_{\pm}(\gm)\right]^n |H_0|^{-1/2}u,|H_0|^{1/2}v\right)\right|\\
=&\left|\sum_{n=0}^{\infty}\begin{pmatrix}-1/2\\n\end{pmatrix}\left(|H_0|^{1/2}\left[-X_{\mp}(\gm)X_{\pm}(\gm)\right]^n |H_0|^{-1/2}u,v\right)\right|\\
\leq &\sum_{n=0}^{\infty}\left|\begin{pmatrix}-1/2\\n\end{pmatrix}\right|\left\|\,|H_0|^{1/2}X_-(\gm)|H_0|^{-1/2}\right\|^n\left\||H_0|^{1/2}X_+(\gm)|H_0|^{-1/2}\right\|^n. 
\end{split}\]
By the definition of the adjoint, this implies that
\[
\Omega_{\pm}(\gm)|H_0|^{-1/2}P_{\pm}\H\subset P_{\pm}\mathcal{Q},
\]
whence, by the closed graph theorem,
\[
 |H_0|^{1/2}\Omega_{\pm}(\gm)|H_0|^{-1/2}\in\LL(P_{\pm}\H).
\]
It is then easy to see that
\beq\label{eq. U bounded on Q}
 |H_0|^{1/2}U(\gm)|H_0|^{-1/2}\in\LL(\H),\quad |H_0|^{1/2}U(\gm)^{*}|H_0|^{-1/2}\in\LL(\H),
\eeq
and we have
\[\begin{split}
 &|H_0|^{-1/2}\widehat{H}_{\diag}(\gm)|H_0|^{-1/2}\\
&=\left(|H_0|^{-1/2}U(\overline{\gm})^*|H_0|^{1/2}\right)\left(|H_0|^{-1/2}H(\gm)|H_0|^{-1/2}\right)\left(|H_0|^{1/2}U(\gm)|H_0|^{-1/2}\right).
\end{split}\]
Since the first and the third factor above are bounded, the claim is proved if the second factor has a bounded closure. Since $\dom(H(\gm))=\dom(H_0)$, it follows by the Heinz inequality that $\dom(|H(\gm)|^{1/2})=\dom(|H_0|^{1/2})$. The latter is equivalent to the boundedness of $|H_0|^{-1/2}H(\gm)|H_0|^{-1/2}$ on $|H_0|^{1/2}\dom(H(\gm))$, see e.g.\ \cite[Theorem 3.2.]{GKMV10Rep}. Since $H(\gm)$ is a core for $|H_0|^{1/2}$, this domain in dense in $\H$, and the operator has a bounded closure. This may also be verified directly by formula \eqref{H form construction}.

\medskip
To prove \eqref{eq. X mapping property}, we introduce the operator
\[
 W:=\mu |H_0|^{-1}P_+-|H_0|^{-1}P_-\in\LL(\H),\quad \mu:=\frac{\|VH_0^{-1}\|}{2-\|VH_0^{-1}\|}<1.
\]
For $u\in\dom(H_0)$, we have
\[\begin{split}
&\re(WH(\gm)u,u)=\re(H(\gm)u,Wu)\geq(H_0u,Wu)-|(Vu,Wu)|\\
&\geq \mu\|P_+u\|^2+\|P_-u\|^2-\|u\| \, \|VH_0^{-1}\| \, \|\mu P_+u-P_-u\|\\
&\geq \mu \,\|P_+u\|^2+\|P_-u\|^2- \frac{1}{2}\cdot \|VH_0^{-1}\| \left(\|u\|^2+\mu^2\,\|P_+u\|^2+\|P_-u\|^2\right)\\
&\geq \left(\mu-\frac{1}{2}\cdot\|VH_0^{-1}\| -\mu\cdot\frac{1}{2}\cdot\|VH_0^{-1}\| \right)\|P_+u\|^2+\left(1-\|VH_0^{-1}\| \right)\|P_-u\|^2,
\end{split}\]
and both summands in the last line are nonnegative. Thus, $H(\gm)$ is $W$-accretive, and $Q_{\pm}(\gm)\H$ are $W$-nonnegative and $W$-nonpositive, respectively, by Theorem~\ref{W-space accretive thm.}. From the $W$-nonnegativity of $Q_+$ e.g.\ it follows that for $u\in P_+\H$,
\[
 \mu\|\,|H_0|^{-1/2}u\|^2-\||H_0|^{-1/2}X_+(\gm)u\|^2=\left(W\begin{pmatrix}u\\X_{+}(\gm)u\end{pmatrix},\begin{pmatrix}u\\X_{+}(\gm)u\end{pmatrix}\right)\geq 0,
\]
or, put differently,
\[
 \|\,|H_0|^{-1/2}X_{+}(\gm)|H_0|^{1/2}u\|\leq \mu\|u\|,\quad u\in P_+\dom(H_0).
\]
Hence, $|H_0|^{-1/2}X_{+}(\gm)|H_0|^{1/2}$ has an extension to an operator in $\LL(P_+\H)$ bounded by $\mu$, and, by duality,
\[
 |H_0|^{1/2}X_+(\gm)^*|H_0|^{-1/2}\in\LL(P_-\H),\quad \|H_0|^{1/2}X(\gm)^*|H_0|^{-1/2}\|\leq \mu<1.
\]
Since $X_+(\gm)^*=-X_-(\overline{\gm})$, half of \eqref{eq. X mapping property} is proved; the other half follows analogously from the nonpositivity of $Q_-$.

\medskip
For each $\gm\in\mathds{D}$, $N\in\N$, we define the bounded forms 
\begin{align*}
\mathfrak{q}^N(\gm)[u,v]&:=\sum_{n=N+1}^{\infty} \frac{\gm^n}{n!}\frac{\rd^n}{\rd\gm^n}\left.\left(T(\gm)u,v\right)\right|_{\gm=0},\quad u,v\in \H,\\
\mathfrak{r}^N(\gm)[u,v]&:=\mathfrak{q}^N(\gm)[|H_0|^{1/2}u,|H_0|^{1/2}v],\quad u,v\in\mathcal{Q}.
\end{align*}
Since $T(\gm)$ is bounded-holomorphic, it clearly holds that
\[
\mathfrak{q}^N(\gm)[u,v]=(Q^{N}(\gm)u,v),\quad u,v\in\H,\quad Q^{N}(\gm):=\sum_{n=N+1}^{\infty} \frac{\gm^n}{n!}\,T^{(n)}(0)\xrightarrow{N\to\infty} 0,
\]
with convergence in the operator norm topology and uniform in $\gm$ on compact subsets of $\mathds{D}$. Observe that
\[
 \widehat{\mathfrak{h}}_{\rm diag}^N(\gm)[u,v]=(\widehat{H}_{\rm diag}(\gm)u,v)-\mathfrak{r}^N(\gm)[u,v],\quad u\in\dom(\widehat{H}_{\rm diag}(\gm)),v\in \mathcal{Q}.
\]

\medskip
We now claim that for $\gm\in(-1,1)$, the operator $\widehat{H}_{\rm diag}(\gm)$ represents the form
\[
\widehat{\mathfrak{h}}_{\rm diag}(\gm)[u,v]:=\mathfrak{h}(\gm)[U(\gm)u,U(\gm)v],\quad u,v\in\dom(\widehat{\mathfrak{h}}_{\rm diag}(\gm)):=U(\gm)^*\mathcal{Q}=\mathcal{Q}
\]
and that $\dom(\widehat{H}_{\rm diag}(\gm))$ is a core for $|H_0|^{1/2}$. This then implies that
\beqnt
 \widehat{\mathfrak{h}}_{\rm diag}^N(\gm)[u,v]=\widehat{\mathfrak{h}}_{\rm diag}[u,v]-\mathfrak{r}^N(\gm)[u,v],\quad u,v\in \mathcal{Q}.
\eeqnt

First, by \eqref{eq. U bounded on Q}, we have $U(\gm)^*\mathcal{Q}=\mathcal{Q}$. Hence, $Y(\gm):=|H_0|^{1/2}U(\gm)^*|H_0|^{-1/2}$ is onto, one-to-one and bounded and thus an isomorphism by the closed graph theorem. It follows that
\[
|H_0|^{1/2}\dom(\widehat{H}_{\rm diag}(\gm))=|H_0|^{1/2}U(\gm)^*\dom(H(\gm))=Y(\gm)|H_0|^{1/2}\dom(H(\gm)),
\]
and the latter is dense in $\H$; this proves the second claim. The first claim follows from the spectral theorem if we observe that the spectral families $\{E_{\lm}(\gm)\}_{\lm\in\R}$ and $\{\widehat{E}_{\lm}(\gm)\}_{\lm\in\R}$ of the self-adjoint operators $H(\gm)$ and $\widehat{H}_{\rm diag}(\gm)$ are related by
\[
\widehat{E}_{\lm}(\gm)=U(\gm)^*E_{\lm}(\gm)U(\gm),\quad \lm\in\R.
\]
In particular, this yields
\beq\label{eq. norms equivalent}
\dom(|\widehat{H}_{\rm diag}(\gm)|^{1/2})=U(\gm)^*\dom(|H(\gm)|^{1/2})=U(\gm)^*\dom(|H_0|^{1/2})=\dom(|H_0|^{1/2}),
\eeq
whence $|H_0|^{1/2}$ is $|\widehat{H}_{\rm diag}(\gm)|^{1/2}$-bounded. Therefore, for $u,v\in\H$, $\|u\|=\|v\|=1$,
\begin{align*}
&|\mathfrak{r}^N(\gm)[|\widehat{H}_{\rm diag}(\gm)|^{-1/2}u,|\widehat{H}_{\rm diag}(\gm)|^{-1/2}v]|\\
=&|\mathfrak{q}^N(\gm)[|H_0|^{1/2}|\widehat{H}_{\rm diag}(\gm)|^{-1/2}u,|H_0|^{1/2}|\widehat{H}_{\rm diag}(\gm)|^{-1/2}v]|\\
\leq&\|Q^{N}(\gm)\| \, \|\,|H_0|^{1/2}|\widehat{H}_{\rm diag}(\gm)|^{-1/2}\|^2 \xrightarrow{N\to\infty} 0
\end{align*}
locally uniformly in $\gm\in\mathds{D}$. Thus, for $N$ sufficiently large, \cite[Theorem 2.11]{Veselic} yields that there exists a unique self-adjoint operator $\widehat{H}_{\rm diag}^N(\gm)$ associated to the form $\widehat{\mathfrak{h}}_{\rm diag}^N(\gm)$.
By construction, $\dom(\widehat{H}_{\rm diag}^N(\gm))$ is a core for $|\widehat{H}_{\rm diag}(\gm)|^{1/2}$. Since the norms $\|\,|\widehat{H}_{\rm diag}(\gm)|^{1/2}\cdot\|$ and $\|\,|H_0|^{1/2}\cdot\|$ are equivalent by \eqref{eq. norms equivalent}, it is also a core for $|H_0|^{1/2}$.

\medskip
The norm-resolvent convergence of $\widehat{H}_{\rm diag}^N(\gm)$ to $\widehat{H}_{\rm diag}(\gm)$ follows from the final argument in the proof of Theorem \ref{th. integral of resolvent difference exists in norm}. More precisely, if $K\subset (-1,1)$ is a compact subset, let
\[
b_N:=\sup_{\gm\in K}\left(\|Q^{N}(\gm)\| \, \|\,|H_0|^{1/2}|\widehat{H}_{\rm diag}(\gm)|^{-1/2}\|^2\right).
\]
We then have
\[
\|(\widehat{H}_{\rm diag}(\gm)-\I\eta)^{-1}-(\widehat{H}^N_{\rm diag}(\gm)-\I\eta)^{-1}\|\leq \frac{b_N}{1-b_N}\frac{1}{|\eta|}\xrightarrow{N\to\infty} 0. \qedhere
\]
\end{proof}

\bibliographystyle{plain}
\bibliography{Block-diagonalization_Arxiv}

\end{document}